\documentclass[12pt]{amsart}

\usepackage{amsmath,amsthm,amssymb,amssymb, paralist, xspace, graphicx, url, amscd, euscript, mathrsfs,stmaryrd,epic,eepic,color}

\usepackage[all]{xy}
\usepackage{amsmath,amscd}
\SelectTips{cm}{}

\usepackage[colorlinks=true]{hyperref}

\usepackage{pstricks}

\numberwithin{equation}{subsection}

1

\numberwithin{subsection}{section}

\allowdisplaybreaks[1]

\newtheorem*{namedtheorem}{\theoremname}
\newcommand{\theoremname}{testing}

\theoremstyle{plain}

\newtheorem{thm}{Theorem}[section]
\newtheorem{proposition}[thm]{Proposition}
\newtheorem{proposition-definition}[thm]{Proposition-Definition}
\newtheorem{lemma-definition}[thm]{Lemma-Definition}
\newtheorem{corollary}[thm]{Corollary}
\newtheorem{lemma}[thm]{Lemma}

\theoremstyle{definition}

\newtheorem{remark}[thm]{Remark}

\theoremstyle{remark}

\numberwithin{thm}{section}

\newcommand\cO{\mathcal{O}}

\def\P{\mathbb{P}}
\def\A{\mathbb{A}}

\newcommand\uH{\underline{H}}

\newcommand\uW{\underline{W}}
\newcommand\uX{\underline{X}}
\newcommand\uY{\underline{Y}}
\newcommand\uZ{\underline{Z}}

\renewcommand\AA{\mathbb{A}}

\newcommand\kk{\mathbf{k}}  

\newcommand\arr{\ifinner\to\else\longrightarrow\fi}

\def\displaytimes_#1{\mathrel{\mathop{\times}\limits_{#1}}}

\def\displayotimes_#1{\mathrel{\mathop{\bigotimes}\limits_{#1}}}

\newdir{ >}{{}*!/-5pt/@{>}}

\newcommand\doublelong[2]{\mathbin{\xymatrix{{}\ar@<3pt>[r]^{#1}
\ar@<-3pt>[r]_{#2}&}}}

\newlength{\ignora}

\renewcommand{\setminus}{\smallsetminus}

\numberwithin{equation}{subsection}

\newcommand{\cok}{\operatorname{Cok}}

\newcommand{\interior}[1]{\overset{\circ}{#1}}

\newcommand{\str}{\operatorname{ {\bf  str} }}

\usepackage{mathrsfs}

\newcommand{\sA}{\mathscr{A}}

\usepackage{comment}

\makeatletter
\newcommand{\leqnos}{\tagsleft@true\let\veqno\@@leqno}
\newcommand{\reqnos}{\tagsleft@false\let\veqno\@@eqno}
\reqnos
\makeatother

\begin{document}


\title[Plane $\A^1$-curves in positive characteristics]{Plane $\A^1$-curves on the complement of strange rational curves}

\author{Qile Chen}

\author{Ryan Contreras}

\address[Q. Chen]{Department of Mathematics\\
Boston College\\
Chestnut Hill, MA 02467\\
U.S.A.}
\email{qile.chen@bc.edu}

\address[R. Contreras]{Department of Mathematics\\
Boston College\\
Chestnut Hill, MA 02467\\
U.S.A.}
\email{ryan.contreras@bc.edu}

\date{\today}
\begin{abstract}
A plane curve is called {\em strange} if its tangent line at any smooth point passes through a fixed point, called the {\em strange point}. In this paper, we study $\A^1$-curves on the complement of a  rational strange curve of degree $p$ in characteristic $p$. We prove the connectedness of the moduli spaces of $\A^1$-curves with given degree, classify their irreducible components, and exhibit the inseparable  $\A^1$-connectedness via the $\A^1$-curves parameterized by each irreducible component. The key to these results is the strangeness of all $\A^1$-curves. As an application, in every characteristic  we construct explicit covering families of $\A^1$-curves, whose total spaces are smooth along large numbers of cusps on each general fiber.
\end{abstract}



\maketitle

\tableofcontents

\section{Introduction}\label{sec:intro}

Throughout this paper, we work over an algebraically closed field $\kk$ of characteristic $p > 0$. 

\subsection{$\A^1$-curves, their moduli, and strangeness}

Consider a pair  $W = (\uW, \Delta_W)$ consisting of a proper variety $\uW$ and a reduced boundary divisor $\Delta_W \subset \uW$. An {\em ${\A^1}$-curve} in $W$ is a non-trivial proper morphism $\interior{f} \colon \A^1 \to \interior{\uW}:=\uW \setminus \Delta_W$. Equivalently, it is a morphism of pairs $f \colon \P^1_{\infty} := (\underline{\P}^1, \infty) \to W$ such that the source is a pre-stable curve with a unique {\em marked point} $\infty \in \underline{\P}^1$, and  $f^{-1}\Delta_W$ is supported entirely on $\infty$. Thus, the local intersection of $f$ against $\Delta_W$ is $\deg f^*\Delta_W$, which is called the {\em contact order} at $\infty$. Note that for a {\em family} of $\A^1$-curves, markings form a section of the source curve.

$\A^1$-curves are the analogue of rational curves for pairs, and their existence is shown to govern the birational and arithmetics geometry of the pair $(\uW, \Delta_W)$, see for example \cite{KM99, Mi01,Cam11, CZ16rankone, CZ18Strongapp}. In this paper, we are interested in $\A^1$-geometry of the pair $X := (\uX = \P^2, \Delta_X)$ with $\Delta_X$ an reduced irreducible curve defined by 
\begin{equation}\label{eq:boundarydefinition}
    \Delta_X :=\Big( \sigma(x_0,x_1)-x_2^p=0 \Big)
\end{equation}
where:
\begin{equation*}
    \sigma(A,B):= \sum_{i=1}^{p-1}\sigma_i A^iB^{p-i}
\end{equation*}
with $\sigma_i\in \kk$ and $\sigma_1=\sigma_{p-1}=1$. Here and throughout this paper, we fix the homogeneous coordinates $[x_0:x_1:x_2]$ of $\uX = \P^2$.  Note that $\Delta_X$ is smooth iff $p=2$, and has cusps for $p\geq 3$. 

Let $\mathscr{A}_{pd}(X)$ be the moduli of $\A^1$-curves in $X$ of degree $d$. As one expects, and will be shown below that $\A^1$-curves in $X$ can be highly obstructed in general, even if $\Delta_X$ is a smooth conic.  Nevertheless we will provide an explicit parameterization of all $\A^1$-curves in $X$, and prove:
\begin{thm}\label{thm:A1-curve-X}
The moduli space $\mathscr{A}_{pd}(X)$ is connected with $\lfloor d/p \rfloor + 1$ non-empty irreducible components:
\[
\mathscr{A}_{pd}(X) = \bigcup_{m}\mathscr{A}_{pd,m}(X)
\]
where $m$ run through all integers satisfying 
\[
0 \leq m \leq d \ \ \mbox{and} \ \ m \equiv d \mod p,
\] 
such that
\begin{enumerate}
 \item $\mathscr{A}_{pd,m}(X)$ has an open dense locus $\mathscr{A}^{\circ}_{pd,m}(X)$ parameterizing $\A^1$-curves with $\str = (0,0,1) \in X$ an ordinary point of multiplicity $m$. 
 \item $\dim \mathscr{A}_{pd,m}(X) = \frac{2}{p} \cdot d + (1-\frac{2}{p}) \cdot  m$.
 \item There is a surjective morphism from an affine space
 \[
 \A^{\dim \mathscr{A}_{pd,m}(X) + 1} \longrightarrow \mathscr{A}_{pd,m}(X).
 \]
\end{enumerate}
\end{thm}
The proof of the above result will be concluded in Section \ref{ss:back-to-X}. The key is the deformation of $\A^1$-curves along a certain foliation $\mathcal{F}_{X/Z^{(p)}}$ on $X$, see Section \ref{ss:X-Z-foliation}. Geometrically, this amounts to the following strangeness of {\em all} $\A^1$-curves, see Lemma \ref{lem:strange-lines} and Proposition \ref{prop:A1-strange}. For general information on foliations in characteristic $p$ see \cite{ekedahl1987foliations} and \cite{miyaoka-book}. 

\begin{proposition}\label{intro-prop:strangeness}
\begin{enumerate}
 \item $\A^1$-lines are precisely the lines through $\str$. 
 \item Let $f(\P^1_{\infty})$ be the reduced image of an $\A^1$-curve $f \colon \P^1_{\infty} \to X$. Then the tangent line at any smooth point of $f(\P^1_{\infty})$ is an $\A^1$-line.
\end{enumerate}
\end{proposition}

A plane curve is called {\em strange} if its tangent line at any smooth point passing through a same point, called the {\em strange point}. It was shown that a degree $p$ reduced and irreducible plane rational curve is strange iff it is projectively equivalent to the form \eqref{eq:boundarydefinition}, see \cite[Theorem 3.4]{homma1987funny} and \cite[Corollary 3.4]{VA91}. In this case, $\Delta_X$ has the strange point $\str = (0,0,1)$. 

The above results imply that all $\A^1$-curves are strange with the same strange point $\str$. Indeed, the foliation $\mathcal{F}_{X/Z^{(p)}}$ at a point of $X\setminus \{\str\}$ is given by the tangent of the $\A^1$-line through that point. Roughly speaking, Proposition \ref{intro-prop:strangeness} means that all $\A^1$-curves are tangent to $\mathcal{F}_{X/Z^{(p)}}$.

%
%
%
%

\begin{remark}
In \cite{CZ16irreducibility}, the moduli of $\A^1$-curves are shown to be irreducible and unirational in many homogeneous situations in characteristic $0$, where $\A^1$-curves are {\it unobstructed}.  The method there is to study degeneration of $\A^1$-curves as stable log maps \cite{Chen, AC, GS}.
\end{remark}

\subsection{$\A^1$-connectedness}

Our study of $\mathscr{A}_{pd}(X)$ is inspired by the $\A^1$-connectedness, see \cite[Definition 1.2]{CZ19A1} or \cite[Definition 9.4]{Cam11}.

The pair $(\uW,\Delta_W)$ is called {\em (separably) $\A^1$-uniruled} if there is a scheme $T$ of $\dim T = \dim W - 1$ and a family of $\A^1$-curves $\interior{f} \colon T\times\A^1 \to \uW \setminus \Delta_W$, such that $\interior{f}$ is dominant (and separable). 

The pair $(\uW,\Delta_W)$ is called {\em (separably) $\A^1$-connected} if there exists a scheme $T$ and a family of $\A^1$-curves $\interior{f} \colon T\times\A^1 \to \uW \setminus \Delta_W$ such that:
\[
\interior{f}  \times_T \interior{f} \colon T\times\A^1\times\A^1 \to  \uW \setminus \Delta_W
\]
is dominant (and separable). These definitions are intrinsic to the interior $\uW \setminus \Delta_W$ and do not depend on the choices of compactifications.

Proposition \ref{intro-prop:strangeness} (1) implies that $X$ is $\AA^1$-uniruled by lines, but inseparably so. Indeed we have the following general  statement for any degree:
\begin{proposition}\label{intro-prop:A1-connectedness}
\begin{enumerate}
 \item For any $p > 0$, $X$ is inseparably $\A^1$-connected by $\A^1$-curves of any non-empty irreducible component $\mathscr{A}_{pd,m}(X)$ with $m < d$. 
  \item $X$ is separably $\A^1$-uniruled iff $p=2$. In this case, $X$ is separably uniruled by $\A^1$-curves in $\mathscr{A}_{2d,0}(X)$ for even $d$. 
\end{enumerate}
\end{proposition}

The proof of the above result will be concluded in Section \ref{ss:back-to-X}. For completeness we note that the moduli space $\sA_{pd,d}(X)$ parameterizes $\AA^1$-curves which are  degree $d$-covers of $\A^1$-lines (Corollary \ref{cor:parameterizing-A1-X}).

For $p \geq 3$, $X$ is not log Fano. In this case, it is expected that $X$ is not separably $\AA^1$-uniruled.

However, $X$ is log smooth and log Fano when $p=2$.  The above statement implies that even in this case, $X$ is not separably $\AA^1$-connected. Indeed, Yi Zhu observes that this is true in a more general situation:

\begin{proposition}\label{intro-prop:yi-observation}
$(\P^n, \Delta)$ is not separably $\A^1$-connected for any smooth hypersurface $\Delta \subset \P^n$ of degree $k$ with $p \mid k$. 
\end{proposition}
\begin{proof}
Consider the residue sequence
\[
0 \to \Omega_{\P^n} \to \Omega_{\P^n}(\log \Delta) \to \cO_\Delta \to 0
\]
Taking the long exact sequence, one obtains: 
\[
H^0(\Omega_{\P^n}(\log \Delta)) \longrightarrow H^0(\cO_\Delta) \stackrel{c}{\longrightarrow} H^1( \Omega_{\P^n})
\]
However, the morphism $c$ taking the divisor class of $\Delta$ is the zero morphism in characteristic $p$ by \cite[Page 37]{MA59}. This implies that  $\dim H^0(\Omega_{\P^n}(\log \Delta)) > 0$. The statement thus follows from the observation in \cite[Corollary 2.8]{CZ19A1} that separable $\A^1$-connectedness implies vanishing of $H^0(\Omega_{\P^n}(\log \Delta))$. 
\end{proof}

In contrast, $(\P^n, \Delta)$ is separably $\A^1$-connected for a general $\Delta$ with  $p \nmid k$ by \cite[Proposition 4.3]{CZ14}. Yi's observation is one of the major motivations for us to understand the geometric behaviors of $\A^1$-curves when boundaries are of degree divisible by $p$. Indeed, as shown in Section \ref{ss:failure-of-SAC}, the failure of separable $\A^1$-connectedness of $X$ is caused by the foliation $\mathcal{F}_{X/Z^{(p)}}$ that destabilizes the {\em log tangent bundle} $T_X$, or more geometrically the strangeness of $\A^1$-curves! 

In a different direction, Proposition \ref{intro-prop:A1-connectedness} implies that $X$ is $\A^1$-connected by free but not very free $\A^1$-curves for $p = 2$. This can be viewed as an example of \cite[Definition 1.2]{Shen10} for $\A^1$-curves, and the existence of $\mathcal{F}_{X/Z^{(p)}}$ is an analogue of \cite[Theorem 1.3]{Shen10}. However, the construction in \cite{Shen10} doesn't apply directly to $\A^1$-curves. 

We plan to study along the lines of foliations/strangeness in higher dimensions in our future work. 

\subsection{Supercuspidal families}

By a {\em supercuspidal family}, we mean a family of rational curves with smooth total space but each fiber has cusps. These families do not exist in characteristic zero by Sard's Theorem \cite[Proposition 7.4]{badescu-algebraic-surfaces}. Examples of supercuspidal families include quasi-elliptic fiberations in $p=2, 3$ \cite{Tate52, BM76}, and Raynaud-Mukai's construction of Kodaira non-vanishings via inseparable covers \cite{Mu13}. The local structures of one dimensional super-cuspidal families are classified in  \cite{Sh92}, which appears in covering families of rational curves on smooth surfaces with non-negative Kodaira dimension. 

In view of Proposition \ref{intro-prop:A1-connectedness}, we study singularities of covering families of $\A^1$-curves in our situation. 
For this purpose, we investigate the components $\sA_{pd,d-p}(X)$ for any $d \geq p > 0$.

\begin{proposition}\label{intro-prop:d-p-equation}
Every $\A^1$-curve in $\sA_{pd,d-p}(X)$ has its image cut out by precisely $\Psi = 0$ for some coefficients $a_j, b_i, c_k \in \kk$, where 
\[
\Psi = 
\begin{cases}
L_1^d-(-1)^d \pi^p\Delta_X, \ &\mbox{for} \ d - p =0; \\
L_1^d-(-1)^d \pi^p\Delta_X L_0 , \ &\mbox{for} \ d - p =1; \\
L_1^d-(-1)^d \pi^p\Delta_X \cdot L_0 \displaystyle \prod_{j=1}^{d-p-1}(L_0-a_j^p L_1), \ &\mbox{for} \ d-p > 1.
\end{cases}
\]
and 
\[
\begin{split}
L_k &:= c_k x_0 - b_k x_1, \ \ \mbox{for} \ \ k = 0, 1, \\
\pi &:= b_1 c_0 - b_0 c_1. 
\end{split}
\]
\end{proposition}
These coefficients $a_j, b_i, c_k \in \kk$ come naturally from the explicit parameterization of $\A^1$-curves in \eqref{eq:basic-polynomial}. 

Fixing a general choice of $a_1, \cdots, a_{d-p-1} \in \kk$ if $d-p > 1$, we define 
\[
\mathcal{C} := (\Psi = 0) \subset \P^2 \times \A^4
\]
where we view $b_0, b_1, c_0, c_1$ as the coordinates of $\A^4$.  Thus $\mathcal{C} \to \A^4$ is a covering family of images of  $\A^1$-curves.

\begin{thm}\label{intro-thm:supercusps}
Let $\mathcal{C}_w$ be a general fiber of $\mathcal{C} \to \A^4$. Then $\str$ is an ordinary point of $\mathcal{C}_w$ of multiplicity $d-p$, and $\mathcal{C}_w$ has only cuspidal singularities away from $\str$. Furthermore, for a general choice of $\Delta_X$,
\begin{enumerate}
 \item  if $p=2$ then $\mathcal{C}_w$ has $d-2$ cusps, and
 \item  if $p>2$ then $\mathcal{C}_w$ has $2d-p-2$ cusps.
\end{enumerate}

Regarding the smoothness of the total space $\mathcal{C}$ along $\mathcal{C}_w$, we have  
\begin{enumerate}
 \item[(a)] If $d-p=0$, then the sigularities of $\mathcal{C}$ along $\mathcal{C}_w$ are precisely the cusps of $\mathcal{C}_w$. 
 \item[(b)] If $d-p = 1$, then $\mathcal{C}$ is smooth along $\mathcal{C}_w$. 
 \item[(c)] If $d-p > 1$, then $\mathcal{C}$ is smooth along $\mathcal{C}_w\setminus \{\str\}$.
\end{enumerate}
\end{thm}

The above result is a summary of Propositions \ref{prop:sing-total-space}, \ref{prop:m=0-p-neq-2-sing}, \ref{prop:p=2-fiber-singularity}, and \ref{prop:cusp-in-general}. 

Thus these families provide explicit examples of supercuspidal families with large numbers of cusps for any $d > p > 0$. However, the theorem also shows that for any $d=p>2$, none of the cusps along a general fiber of $\mathcal{C} \to \A^4$ is ``super''. Our result on the supercuspidal geometry  is based on explicit computation of this component.  It would be interesting to understand the geometric reason behind such mixed behavior of $\A^1$-curves, for studying other components and targets in general. 


\begin{remark}
We note that $\A^1$-curves are special in general in the moduli of strange curves with fixed strange point. A general degree $d$ strange curve with strange point $\str$ and $m=d-p$ is rational and they form family of dimension $2d-p+3$ \cite{VA91}, which is larger than $\dim \sA_{dp,d-p}(X)=d-p+2$.

Despite being special, the number of cusps in Theorem \ref{intro-thm:supercusps} of a general $\A^1$-curves are equal to the number of cusps on a general rational strange curve  of the same degree \cite[(7.2), (7.3)]{VA91}. 
\end{remark}

\subsection{Acknowledgements}
The authors would like to thank Yi Zhu for lots of inspiring and fruitful discussions.  We also benefit from discussions with Kuan-Wen Lai, Brian Lehmann, and Dawei Chen.

Research by Qile Chen and Ryan Contreras was supported in part by NSF grants DMS-1700682 and DMS-2001089.

\section{The strangeness of $\A^1$-curves}\label{sec:strange}

\subsection{The foliation $X \to Z^{(p)}$}\label{ss:X-Z-foliation}
 Let $Z$ be the log scheme associated to the pair $(\uZ = \P(1,1,p), \Delta_Z)$ with boundary $\Delta_Z = (z_2 = 0)$ where we fix the homogeneous coordinates $(z_0, z_1, z_2)$ of $Z$ with weights $(1,1,p)$. Consider the Frobenius morphisms of log schemes $Z \to Z^{(p)}$ over $\kk$ whose underlying is given by the usual Frobenius over $\kk$. Denote by $[z^{p}_0:z^{p}_1:z^{p}_2]$ the homogeneous coordinates of $Z^{(p)}$. We define a sequence of morphisms of log schemes as follows:
 \begin{equation}\label{eq:X-Z-foliation}
\xymatrix{
Z \ar[rr]^{F_{Z/X}} && X \ar[rr]^{F_{X/Z^{(p)}}} && Z^{(p)}
}
 \end{equation}
such that:
\[
\begin{split}
 F_{X/Z^{(p)}}^* &\colon [z^p_0: z^p_1: z^p_2] \mapsto [x_0: x_1 : \sigma(x_0, x_1) - x^p_2], \\
 F_{Z/X}^* &\colon [x_0 : x_1 : x_2] \mapsto [z_0^p : z_1^p : \sigma^{1/p}(z_0, z_1) - z_2],
\end{split}
\] 
where we define
$
  \sigma^{1/p}(A,B):= \sum_{i=1}^{p-1}\sigma_i^{1/p}A^iB^{p-i}.
$ 

We note that the first morphism $F_{Z/X}$ is an inseparable $p$-cyclic ramified ``along" the boundary divisor $\Delta_X$. In \cite{kollar1995nonrational} Kollar uses similar objects to produce singular Fano varieties which are not ruled. 

The morphisms in \eqref{eq:X-Z-foliation} are well-defined on the level of underlying schemes. Furthermore, we check that the composition $F_{X/Z^{(p)}} \circ F_{Z/X}$ is the Frobenius morphism $Z \to Z^{(p)}$ over $\kk$. To see that each one induces a morphism on the log schemes level, it suffices to observe that
\begin{equation}\label{eq:X-Z-pull-back-boundary}
F_{X/Z^{(p)}}^*\Delta_{Z^{(p)}} = \Delta_{X} \ \ \ \mbox{and} \ \ \ F_{Z/X}^*\Delta_{X} = p\cdot\Delta_Z.
\end{equation}
where the second pull-back follows from 
\[
F_{Z/X}^*(\sigma(x_0,x_1) - x^p_2) = \sigma(z_0^p,z_1^p) - \Big(\sigma^{1/p}(z_0,z_1) - z_2 \Big)^p = z_2^p.
\]
Note that $Z$ has a unique singularity at $\str_Z := [0:0:1]$, whose image in $X$ is $\str$. 

We view $X=(\uX,\Delta_X)$ as the log scheme with the Deligne-Faltings log structure associated to the Cartier divisor $\Delta_X$ \cite[Complement 1]{KKato}. Note that $X$ is log smooth, in the sense of \cite[(3.5)]{KKato}, away from the singularities of $\Delta_X$ where the log smoothness of $X$ fails. Let $\Omega_{X}$ be the log cotangent sheaf of $X$, it consists of differentials with logarithmic poles along $\Delta_X$, and is locally free over the log smooth locus of $X$. Denote by $T_X := \Omega_X^{\vee}$ the log tangent sheaf. The inclusion $\Omega_{\uX} \subset \Omega_X$ implies $T_X \subset T_{\uX}$ as a subsheaf.  Since $T_X$ is reflexive over the smooth surface $\uX$, it is locally free. Let $\mathcal{F}_{X/Z^{(p)}} \subset T_{\uX}$ be the foliation inducing the underlying of $F_{X/Z^{(p)}}$. 

\begin{lemma}\label{lem:X-Z-foliation}
$\mathcal{F}_{X/Z^{(p)}}$ is logarithmic in the sense that $\mathcal{F}_{X/Z^{(p)}} \subset T_X$. Furthermore, we have $\mathcal{F}_{X/Z^{(p)}} \cong \cO_{\P^2}(1)$.
\end{lemma}
\begin{proof}
By the first pull-back in \eqref{eq:X-Z-pull-back-boundary}, the restriction $F_{X/Z^{(p)}}|_{\Delta_X} \colon \Delta_X \to \Delta_{Z^{(p)}}$ is purely inseparable of degree $p$. Thus vector fields in $\mathcal{F}_{X/Z^{(p)}}$ are tangent to $\Delta_X$ by \cite[Proposition 1]{RS76}. This implies  $\mathcal{F}_{X/Z^{(p)}} \subset T_X$.  

For the second statement, consider $\uW = \P(1,1,p)$ with the corresponding homogeneous coordinates $[w_0:w_1:w_2]$, and an isomorphism 
$\theta \colon \uW \to \uZ^{(p)}$ defined by 
\[
\theta^* \colon [z^p_0 : z^p_1 : z^p_2] \mapsto [w_0 : w_1 : \sigma(w_0, w_1) - w_2]. 
\]
We obtain a commutative diagram
\[
\xymatrix{
\uX \ar[rr]^{F_{X/W}} \ar[rrd]_{{F}_{X/Z^{(p)}}} && \uW \ar[d]^{\theta} \\
&& \uZ^{(p)}
}
\]
where $F_{X/W}$ is defined by $F_{X/W}^* \colon [w_0:w_1:w_2] \mapsto [x_0:x_1:x^p_2]$. Let $\mathcal{F}_{X/W}$ be the foliation inducing $F_{X/W}$. Note that the foliations $\mathcal{F}_{X/W}$ and $\mathcal{F}_{X/Z^{(p)}}$ have a unique singularity at $\str$. The above commutative diagram induces a commutative diagram of solid arrows with exact rows
\[
\xymatrix{
0 \ar[r] & \mathcal{F}_{X/Z^{(p)}}|_{\uX\setminus \{\str\}}  \ar[rr] \ar@{-->}[d] && T_{\uX}|_{\uX\setminus \{\str\}} \ar[d]^{=} \ar[rr] && T_{\uZ^{(p)}}|_{\uX\setminus \{\str\}} \ar[d]^{\cong} \\
0 \ar[r] & \mathcal{F}_{X/W}|_{_{\uX\setminus \{\str\}}}  \ar[rr] && T_{\uX}|_{\uX\setminus \{\str\}} \ar[rr] && T_{\uW}|_{\uX\setminus \{\str\}}
}
\]
Thus we obtain $\mathcal{F}_{X/Z^{(p)}} \cong \mathcal{F}_{X/W}$ given by the induced dashed arrow. The second statement follows from $\mathcal{F}_{X/W} \cong \cO(1)$ which is calculated in \cite[Example 2.1]{Saw16}.
\end{proof}

An immediate consequence of Lemma \ref{lem:X-Z-foliation} is the following

\begin{proposition}\label{prop:lift-to-Z}
Every $\A^1$-curve $f \colon \P^1_{\infty} \to X$  lifts to a unique $\A^1$-curve $\tilde{f} \colon \P^1_{\infty} \to Z$. Furthermore, if the contact order of $f$ is $pd$, then the contact order of $\tilde{f}$ is $d$. 
\end{proposition}
\begin{proof}
We may assume that $f$ is birational onto its image. Otherwise, $f$ is given by the composition $\P^1_{\infty} \longrightarrow \P^1_{\infty} \stackrel{f'}{\longrightarrow} X$ where the first arrow is a cover totally ramified at least at $\infty$, and the second arrow is a $\A^1$-curve birational onto its image. Thus to lift $f$, it suffices to consider the lift of $f'$. 

Suppose $f^*T_{X}$ has the splitting type with $a \geq  b$:
\[
f^*T_{X}  \cong \cO(a)\oplus \cO(b).
\] 
Since $\Delta_{X}$ is of degree $p \geq 2$, we have $a + b = (3-p)d$. Consider the inclusion $f^*\mathcal{F}_{X/Z^{(p)}} \subset f^*T_X$. By Lemma \ref{lem:X-Z-foliation}, we have $f^*\mathcal{F}_{X/Z^{(p)}} \cong \cO(d)$ where $d \geq 1$ is the degree of $f$. Thus we necessarily have 
\begin{equation}\label{eq:not-ample}
a \geq  d > 0 \geq b ,
\end{equation}
hence $f^*\mathcal{F}_{X/Z^{(p)}} \subset \cO(a)$.   

On the other hand, since $f$ is birational onto its image by our assumption, we have a non-trivial morphism $d f \colon T_{\P^1_{\infty}} \to f^*T_X$. Since $T_{\P^1_{\infty}} \cong \cO(1)$, $df$ factors through a non-trivial morphism $T_{\P^1_{\infty}} \to \cO(a)$. This means that the foliation $\mathcal{F}_{X/Z^{(p)}}$ is tangent to the image $f(\P^1_{\infty})$. By \cite[Proposition 1]{RS76}, the composition $\P^1_{\infty} \stackrel{f}{\longrightarrow} X \to Z^{(p)}$ factors through $\tilde{f}^{(p)} \colon (\P^1_{\infty})^{(p)} \to Z^{(p)}$, which yields the unique lift $\tilde{f} \colon \P^1_{\infty} \to Z$ as needed. 

The second statement follows from the second equation in \eqref{eq:X-Z-pull-back-boundary} and the projection formula.
\end{proof}

\subsection{Tangents of $\A^1$-curves}\label{ss:strangeness}
The lifting property in Proposition \ref{prop:lift-to-Z} is closely related to the strangeness of $\A^1$-curves. We first observe:

\begin{lemma}\label{lem:strange-lines}
 $\A^1$-lines in $X$ are precisely the lines through $\str$. 
\end{lemma}
\begin{proof}
By Proposition \ref{prop:lift-to-Z}, $\A^1$-lines in $X$ lifts to $\A^1$-curves in $Z$ of contact order $1$. Note that any curve $L \subset Z$ with the intersection number $\Delta_Z \cap L = 1$ necessarily passing through the singularity $\str_Z$ of $Z$ whose image in $X$ is $\str$. Thus,  every $\A^1$-lines in $X$ contain $\str$. 

On the other hand, a line $L$ through $\str$ is of the form $ax_0 + b x_1 = 0$. By \eqref{eq:boundarydefinition}, it is straight forward to check that $L \cap \Delta_X$ is supported at a single point. 
\end{proof}

We show the all $\A^1$-curves in $X$ are strange in the following sense.

\begin{proposition}\label{prop:A1-strange}
For any $\A^1$-curve $f \colon \P^1_{\infty} \to X$, denote by $D \subset \uX$ the image $f(\P^1_{\infty})$ with the reduced structure. Then at any smooth point $x$ of  $D$, the tangent line of $D$ is the $\A^1$-line through $x$. In particular, $D$ is strange.  
\end{proposition}
\begin{proof}
By Proposition \ref{prop:lift-to-Z}, the foliation $\mathcal{F}_{X/Z^{(p)}}$ is tangent to both the $\A^1$-line through $x$ and $D$ at $x$. The strangeness of $D$ then follows from Lemma \eqref{lem:strange-lines}. 
\end{proof}

This completes the proof of Proposition \ref{intro-prop:strangeness}. \qed

\subsection{Failure of separable $\A^1$-connectedness}\label{ss:failure-of-SAC}

Recall that for the pair $W=(\uW,\Delta_W)$ the (separable) $\A^1$-uniruledness and $\A^1$-connectedess are intrinsic to the open part $\uW \setminus \Delta_W$. It is more convenient to check these properties using compacitifications by adding the boundary $\Delta_W$. Suppose that $W$ is log smooth in the sense of \cite[(3.5)]{KKato}, with the log tangent bundle  $T_W$. Similar to the situation of rational connectedness, under the log smoothness assumption, $W$ is separably $\A^1$-uniruled (resp. separably $\A^1$-connected) iff there is an $\A^1$-curve $f \colon \P^1_{\infty} \to W$ such that $f^*T_W$ is semi-positive (resp. ample), see \cite[Proposition 2.5, 2.6]{CZ19A1}. In this case we call $f$ a {\em free} (resp. {\em very free}) $\A^1$-curve.  In particular, deformations of (very) free curves are unobstructed. 

In our situation $\AA^1$-curves are highly obstructed in general due to the existence of the logarithmic foliation $\mathcal{F}_{X/Z^{(p)}}$.


\begin{lemma}\label{lem:non-freeness}
For any $\AA^1$-curve $f \colon \P^1_{\infty} \to X$ the pullback $f^\ast T_X$ is not ample.  Furthermore, $f^*T_X$ is semi-positive iff $p=2$ and $\str \not\in f(\P^1_{\infty})$, in which case $f^*T_X = \cO(d)\oplus \cO$. 
\end{lemma}
\begin{proof}
Suppose we have the splitting type $f^*T_{X}  \cong \cO(a)\oplus \cO(b)$ with $a \geq b$ as in the proof of Proposition \ref{prop:lift-to-Z}. Then $0\geq b$ as in \eqref{eq:not-ample} implies that $f^*T_{X}$ not ample. 

Suppose $b = 0$. Then $a + b = (3-p)d \geq 0$ and $a \geq d > 0$ implies that $p = 2$. In this case $X$ is log smooth and we have an exact sequence (\cite[Proposition 0.3.18]{dolg2020enriques}): 
\[
0 \to \mathcal{F}_{X/Z^{(p)}} \to T_X \to \mathcal{I}_{\str} \to 0.
\]
where the quotient $\mathcal{I}_{\str}$ is the ideal sheaf of $\str$ since $X$ is log smooth, and $\mathcal{F}_{X/Z^{(p)}}$ has a unique singularity at $\str$. Pulling back the above exact sequence along $f$, we obtain 
\[
0 \to \cO(d) \to \cO(a)\oplus\cO(b) \to \cO \to 0. 
\]
over $\P^1_{\infty}$ with $a + b = d$. Since $\operatorname{Ext}^1_{\P^1}(\cO,\cO(d)) = 0$, we have $a=d$ and $b = 0$, as needed. 
\end{proof}

\begin{proposition}\label{prop:SAC-failure}
\begin{enumerate}
 \item $X$ is not separably $\A^1$-connected. 
 \item $X$ is separably $\A^1$-uniruled iff $p=2$.  
\end{enumerate}
\end{proposition}
\begin{proof}
When $p=2$, $X$ is log smooth. In this case, the separable $\A^1$-connected (resp. separable $\A^1$-uniruleness) is equivalent to the existence of very free (resp. free) $\A^1$-curves \cite[Proposition 2.7]{CZ19A1}. Thus, the statement follows from Lemma \ref{lem:non-freeness} for $p=2$. 

Now we assume $p > 2$, in which case $X$ fails to be log smooth precisely along the singularities of $\Delta_X$.  Observe that separable $\A^1$-connectedness implies separable $\A^1$-uniruledness of $X$ by definition. Thus it suffices to verify that $X$ is not separably $\A^1$-uniruled. 

Otherwise, there is a separably dominant morphism $\interior{f}\colon S\times \A^1 \to \interior{X}$ with $\dim S = 1$. Similar to \cite[Proposition 2.5]{CZ19A1}, after possibly shrinking $S$, we may extend $\interior{f}$ to a family of $\A^1$-curves $f \colon S \times \P^1_{\infty} \to X$ of degree $d$, and we may further assume that $S$ is smooth and affine. 



Pulling back  along $f$, we obtain a morphism of exact sequences
\begin{equation}\label{diag:not-uniruled}
\xymatrix{
0 \ar[r] & T_{\P^1_{\infty}} \ar[r] \ar[d] &  T_S\oplus T_{\P^1_{\infty}} \ar[r] \ar[d]_{d f} &  T_S \ar[d] \ar[r] & 0 \\
0 \ar[r] & f^*\mathcal{F}_{X/Z^{(p)}} \ar[r]  &  f^*T_{X} \ar[r] &  f^*Q \ar[r] & 0 
}
\end{equation}
By Lemma \ref{lem:non-freeness}, the bottom of \eqref{diag:not-uniruled} over $s \in S$ is of the form
\[
0 \to  \cO(d) \to \cO(a_s)\oplus \cO(b_s) \to f^*Q_s \to 0
\]
for $a_s \geq d > 0 > b_s$. Thus we obtain 
\[
f^*Q_s \cong \cO(b_s) \oplus \cok\Big(\cO(d) \to \cO(a_s)\Big).
\]

For any vector field $v \in H^0(T_S)$, the image of $df(v)|_{s}$ in $f^*Q$ is contained the torsion $\cok(\cO(d) \to \cO(a_s))$. This implies that on the fiber level
\begin{equation}\label{eq:general-vf}
df(v)|_{(s,z)} \in \left(f^*\mathcal{F}_{X/Z^{(p)}}\right)\Big|_{(s,z)}
\end{equation}
for a general point $(s,z) \in S\times \P^1_{\infty} $. Since by assumption $T_S$ is globally generated, $df$ is not generically isomorphic. Hence $f$ is not separable, a contradiction! 
\end{proof}

\section{Parameterizing $\A^1$-curves in $Z$}\label{sec:parameterization}

By Proposition \ref{prop:lift-to-Z}, all $\A^1$-curves in $X$ comes from $\A^1$-curves in $Z$. Indeed, for any $\A^1$-curve $\tilde{f} \colon \P^1_\infty \to Z$ the composition $\P^1_\infty \to Z \to X$ is an $\A^1$-curve in $X$.  In this section, we will provide an explicit parameterization for {\em any} $\A^1$-curves in $Z$. The key is again studying deformation of $\A^1$-curves along certain foliations.

\subsection{The foliation $Y \to Z$}
Let $Y$ be the log scheme associated to the pair $(\uY = \P^2, \Delta_Y)$ with boundary $\Delta_Y = (y_2 = 0)$ where we fix the homogeneous coordinates $(y_0,y_1,y_2)$ of $\P^2$. Note that $Y$ is smooth and log smooth. We again have the Frobenius morphism $Y \to Y^{(p)}$ over $\kk$ given by $[y_0:y_1:y_2] \mapsto [y^p_0:y^p_1:y^p_2]$ which factors as the composition of the two morphisms of log schemes

\begin{equation}\label{eq:Y-Z-foliation}
\xymatrix{
Y \ar[rr]^{F_{Y/Z}} && Z \ar[rr]^{F_{Z/Y^{(p)}}} \ar[rr] && Y^{(p)}
}
\end{equation}
such that
\[
\begin{split}
& F_{Z/Y^{(p)}}^* \colon [y^p_0: y^p_1: y^p_2] \mapsto [z^p_0: z^p_1 : z_2], \\
& F_{Y/Z}^* \colon [z_0 : z_1 : z_2] \mapsto [y_0 : y_1 : y^p_2],
\end{split}
\] 
To see that both $F_{Z/Y^{(p)}}$ and $F_{Y/Z}$ induces morphisms on the log schemes level, we observe that
\begin{equation}\label{eq:Y-Z-pull-back-boundary}
F_{Z/Y^{(p)}}^*\Delta_{Y^{(p)}} = \Delta_{Z} \ \ \ \mbox{and} \ \ \ F_{Y/Z}^*\Delta_{Z} = p\cdot\Delta_Y.
\end{equation}

%

Combining \eqref{eq:X-Z-foliation} and \eqref{eq:Y-Z-foliation}, we obtain a sequence of foliations 
\[
Y \to Z \to X.
\]
This implies that $X$ is $\A^1$-connected in a stronger sense: 

\begin{corollary}\label{cor:X-A1-connected}
There is an $\A^1$-curve in $\interior{X}:=X\setminus \Delta_X$ through any given $n$ points for arbitrary $n > 0$.
\end{corollary}
\begin{proof}
Any given $n$ points in $X$ lift to $n$ points in $Y$. The statement follows from the fact that $\interior{Y} = \A^2$ admits an $\A^1$-curve through any given $n$ points.
\end{proof}

\subsection{Resolving the foliation}
To study the deformation of $\A^1$-curves along $Y \to Z$, we resolve the singularity of $Z$.  Consider the blowup 
$H \to Z$ at $\str_Z$. We equip $H$ with the log structure pulled back from $Z$. Explicitly, $H$ is the log scheme of the pair $(\underline{H}, \Delta_{H})$ where $\Delta_{H}=\Delta_Z$.

Similarly, consider the blowup $H_{Y^{(p)}} \to Y^{(p)}$ at $\str_{Y^{p}} := (0,0,1)$, and equip $H_{Y^{(p)}}$ with the log structure pulled back from $Y^{(p)}$. Further denote by $E_Z \subset H$ and $E_{Y^{(p)}} \subset H_{Y^{(p)}}$ the exceptional curves. 

\begin{lemma}\label{lem:resolve-foliation}
There is a commutative diagram
\begin{equation}\label{diag:resolve-foliation}
\xymatrix{
Z \ar[rr]^{F_{Z/Y^{(p)}}} && Y^{(p)} \\
H \ar[rr]^{F_{H/H_{Y^{(p)}}}} \ar[u] \ar[d] && H_{Y^{(p)}} \ar[u] \ar[d] \\
E_Z \ar[rr]^{F_{E}} && E_{Y^{(p)}}
}
\end{equation}
where the two top vertical arrows are the contraction of the corresponding exceptions curves, the two bottom vertical arrows are the projections, the bottom horizontal arrow is the Frobenius morphism, and the lower square is Cartesian. 
\end{lemma}
\begin{proof}
To see that the middle arrow is well-defined making the upper square commutative, it suffices to consider the neighborhood of the exceptional curves. Let $U_1$ and $U_2$ be the Zariski neighborhood of $E_Z$ with coordinate rings
\[
\kk[U_1] = \kk[\frac{z^p_0}{z_2}, \frac{z_1}{z_0}] \ \ \ \mbox{and} \ \ \ \kk[U_2] = \kk[\frac{z^p_1}{z_2}, \frac{z_0}{z_1}]
\]
where $E_Z$ is defined by $\frac{z^p_0}{z_2} = 0$ on $U_1$, and $\frac{z^p_1}{z_2} = 0$ on $U_2$. 

Let $V_1$ and $V_2$ be the Zariski neighborhood of $E_{Y^{(p)}}$ with coordinate rings
\[
\kk[V_1] = \kk[\frac{y^p_0}{y^p_2}, \frac{y^p_1}{y^p_0}]  \ \ \ \mbox{and} \ \ \ \kk[V_2] = \kk[\frac{y^p_1}{y^p_2}, \frac{y^p_0}{y^p_1}].
\]
where $E_{Y^{(p)}}$ is defined by $\frac{y^p_0}{y^p_2} = 0$ on $V_1$, and $\frac{y^p_1}{y^p_2} = 0$ on $V_2$. Then the definition of the second arrow in \eqref{eq:Y-Z-foliation} implies that the arrow $F_{H/H_{Y^{(p)}}}$ is defined on each chart as follows:
\begin{equation}\label{eq:resolve-foliation}
\begin{split}
U_1 \to V_1, & \ \ \ \left( \frac{z^p_0}{z_2}, \frac{z_1}{z_0} \right) \mapsto \left( \frac{z^p_0}{z_2}, \Big(\frac{z_1}{z_0}\Big)^p \right),\\
U_2 \to V_2, & \ \ \ \left(\frac{z^p_1}{z_2}, \frac{z_0}{z_1}\right) \mapsto \left(\frac{z^p_1}{z_2}, \Big(\frac{z_0}{z_1}\Big)^p \right).
\end{split}
\end{equation}

The morphisms described in \eqref{eq:resolve-foliation} imply that $F_{H/H_{Y^{(p)}}}|_{E_Z}$ is the Frobenius morphism $E_{Z} \to E_{Y^{(p)}}$, and the following is Cartesian
\[
\xymatrix{
U_1 \cup U_2 \ar[rr]  \ar[d] && V_1 \cup V_2 \ar[d] \\
E_Z \ar[rr]^{F_{E}} && E_{Y^{(p)}}
}
\]
Note that $U_1 \cup U_2 = H \setminus \Delta_{Z}$ and $V_1 \cup V_2  = H_{Y^{(p)}} \setminus \Delta_{Y^{(p)}}$. Finally, it suffices to observe that the morphism $\Delta_{Z} \to \Delta_{Y^{(p)}}$ is the Frobenius by \eqref{eq:Y-Z-pull-back-boundary}.
\end{proof}

We next describe the structure of the log tangent bundle $T_{H}$ of $H$.

\begin{proposition}\label{prop:H-foliation}
Let $\mathcal{F}_H \subset T_{\uH}$ be the foliation inducing the underlying of $F_{H/H_{Y^{(p)}}}$. Let $\ell$ be the fiber class of the projection $H \to E_Z$. Then 
\begin{enumerate}
 \item $\mathcal{F}_H$ is logarithmic in the sense that $\mathcal{F}_H \subset T_H$. 
 \item $\mathcal{F}_H \cong \cO(2\ell)$ and there is a natural splitting $T_H = \mathcal{F}_H \oplus \cO(E_{Z})$. 
\end{enumerate}
\end{proposition}

\begin{proof}
For (1), it suffices to observe that $\mathcal{F}_H$ is tangent to $\Delta_{H}$, since the restriction $F_{Z/Y^{(p)}}|_{\Delta_H}$ is the Frobenius morphism. 

For (2), the Cartesian square in \eqref{diag:resolve-foliation} leads to a commutative diagram of solid arrows with exact rows:
\[
\xymatrix{
&&& \mathcal{F}_{H} \ar@{^{(}->}[d] &&& \\
0 \ar[r] & T_{H/E_Z} \ar[rr] \ar[d]^{\cong} && T_H \ar[rr] \ar[d] \ar@{-->}[lld] && T_{E_Z}|_{H} \ar[d]^{0} \ar[r] & 0 \\
0 \ar[r] & T_{H_{Y^{(p)}}/E_{Y^{(p)}}}|_{H} \ar[rr] && T_{H_{Y^{(p)}}}|_H \ar[rr]  && T_{E_{Y^{(p)}}}|_{H}  \ar[r] & 0 
}
\]
The commutativity of the right square implies $T_H \to T_{E_{Y^{(p)}}}|_{H}$ is the zero morphism, hence the dashed arrow shown above. This leads to $T_H \to T_{H/E_Z}$ that splits the middle sequence so that 
\[
T_H \cong T_{H/E_Z} \oplus T_{E_Z}|_{H}.
\]

Next we observe that $\mathcal{F}_H \cap T_{H/E_Z} = 0$ viewed as subsheaves of $T_H$. This is because that $\mathcal{F}_H$ given by the Frobenius base change is nowhere tangent to be fiber of $H \to E_Z$. This leads to the inclusion $\mathcal{F}_H \subset T_{E_Z}|_{H}$ as subsheaves of $T_{H}$. But  $\mathcal{F}_H$ as a foliation is saturated in $T_H$, hence $\mathcal{F}_H = T_{E_Z}|_{H}$. 
Finally, we observe that $T_{H/E_Z} \cong \cO(E_{Z})$ and $T_{E_Z}|_{H} \cong \cO(2\ell)$. This proves (2). 
\end{proof}

\subsection{Deforming $\A^1$-curves in $Z$}

Let $f \colon \P^1_{\infty} \to Z$ be an $\A^1$-curve in $Z$ with the contact order $d$ and multiplicity $m \geq 0$ at $\str_Z$. We lift $f$ to an $\A^1$-curve $g \colon \P^1_{\infty} \to H$ by taking the proper transform of the image of $f$ along $H \to Z$. Instead of deforming $\A^1$-curves in $Z$, we may deform their lifts to $H$ using the following lemma.

\begin{lemma}\label{lem:lift-in-H}
Notations as above, we have 
\begin{enumerate}
\item $m \equiv d \mod p$.
\item $g_{*}[\P^1_{\infty}] = \frac{d-m}{p} \Delta_{H} + m \ell$.
\item $g^*T_H \cong \cO(2\cdot \frac{d-m}{p})\oplus \cO(m)$. 
\end{enumerate}
\end{lemma}
\begin{proof}
Note that (2) implies (1) since the coefficient $\frac{d-m}{p}$ is necessarily an integer. Furthermore (3) follows from (2) and Proposition \ref{prop:H-foliation} (2). 

Setting $g_{*}[\P^1_{\infty}] = a \Delta_{H} + b \ell$, (2) follows from solving the following integral equations:  
\[
\begin{split}
d &= (a \Delta_{H} + b \ell)\cdot \Delta_{H} = a p + b, \\
m&= (a \Delta_{H} + b \ell)\cdot E_Z = b.
\end{split}
\]
\end{proof}

Denote by $\mathscr{A}_{d,m}(H)$ the moduli of $\A^1$-curves in $H$ with curve class $\Big(\frac{d-m}{p} \Delta_{H} + m \ell\Big)$. Lemma \ref{lem:lift-in-H} (3) shows that $g$ is unobstructed \cite[5.9]{olsson-log-cotangent-complex}, hence $\mathscr{A}_{d,m}(H)$ is smooth of the expected dimension:
\begin{equation}\label{eq:d-m-H-dim}
\dim \mathscr{A}_{d,m}(H) = \dim H^0(g^*T_H) - 2 =  2\cdot \frac{d-m}{p} + m=\frac{2d}{p} +\left(1-\frac{2}{p} \right)m,
 \end{equation}
where the ``$-2$'' in the middle takes into account the dimension of the automorphisms of $\P^1_{\infty}$ fixing $\infty$. 

Denote by $\mathscr{A}_d(Z)$  the moduli of $\A^1$-curves in $Z$ of contact order $d$. The contraction $H \to Z$ induces a natural injective morphism:
\[
\mathscr{A}_{d,m}(H) \hookrightarrow \mathscr{A}_d(Z)
\]
Denote its image by $\mathscr{A}^{\circ}_{d,m}(Z)$, and let $\mathscr{A}_{d,m}(Z)$ be its closure in $\mathscr{A}_d(Z)$. By \eqref{eq:d-m-H-dim}, it is of pure dimension
\begin{equation}\label{eq:d-m-Z-dim}
\dim \mathscr{A}_{d,m}(Z) = \dim \mathscr{A}_{d,m}(H) = \frac{2d}{p} + (1-\frac{2}{p}) m.
\end{equation}

\begin{lemma}\label{lem:components-Z}
For $m \neq m'$, suppose $\mathscr{A}_{d,m}(Z)$ and $\mathscr{A}_{d,m'}(Z)$  are non-empty. Then 
\[
\mathscr{A}_{d,m}(Z) \not\subset \mathscr{A}_{d,m'}(Z) \ \ \mbox{and} \ \ \mathscr{A}_{d,m'}(Z) \not\subset \mathscr{A}_{d,m'}(Z).
\]
In particular, $\mathscr{A}_{d,m}(Z)$ and $\mathscr{A}_{d,m'}(Z)$ are different components of $\mathscr{A}_d(Z)$.
\end{lemma}
\begin{proof}
We may assume that $m < m'$. By dimension count \eqref{eq:d-m-Z-dim}, we have $\mathscr{A}_{d,m}(Z) \not\subset \mathscr{A}_{d,m'}(Z)$. On the other hand, the multiplicity at $\str_Z$ is upper semi-continuous \cite[Proposition 5.10]{ito2020deformations}. It  follows that $\mathscr{A}_{d,m'}(Z) \not\subset \mathscr{A}_{d,m}(Z)$.
\end{proof}

In the $m=0$ case, the situation is much simpler thanks to the following result. 

\begin{proposition}\label{prop:lift-to-Y}
Notations as above, suppose  $m =0$. Then $f$ unique lifts to an $\A^1$-curve $\tilde{f} \colon \P^1_{\infty} \to Y$. 
\end{proposition}
\begin{proof}
The proof is similar to Proposition \ref{prop:lift-to-Z}, and we may assume that $f$ is birational onto its image. 

Using $m=0$ and Lemma \ref{lem:lift-in-H} (2), the non-zero differential  $dg \colon \cO(1) \cong T_{\P^1_{\infty}} \to g^*T_H$ has its image contained in the factor $\cO(2\cdot \frac{d-m}{p}) = g^*\mathcal{F}_{H}$. Hence $g$ lifts to $\tilde{g} \colon \P^1_{\infty} \to H_Y$ by Proposition \ref{prop:H-foliation}. The composition $\P^1_{\infty} \to H_Y \to Y$ is the lift $\tilde{f}$ as in the statement.
\end{proof}

\subsection{Parameterizing $\A^1$-curves in $\mathscr{A}_{d,m}(Z)$}\label{ss:parameterization}
Consider the coordinate ring 
\[
\kk[\interior{Z}] = \kk[\frac{z^p_0}{z_2}, \frac{z^{p-1}_0 z_1}{z_2}, \cdots, \frac{z_0 z^{p-1}_1}{z_2}, \frac{z^{p}_1}{z_2}].
\]
For simplicity, we write
\[
U_i := \frac{z^{p-i}_0 z^{i}_1}{z_2}, \ \ \mbox{for} \ i = 0, 1, \cdots, p.
\]
As generators in the toric ring $\kk[\interior{Z}]$, they satisfy precisely the relations
\begin{equation}\label{eq:binomial-relation}
U_i \cdot U_j = U_{i'} \cdot U_{j'}, \ \ \mbox{for any} \ i + j = i' + j'.
\end{equation}

Fix a pair of integers $(d,m)$ such that $d > 0, d \geq m \geq 0$, and $m \equiv d \mod p$. Note that otherwise $\mathscr{A}_{d,m}(Z)$ is empty by Lemma \ref{lem:lift-in-H} (1). To parameterize $\A^1$-curves in $\mathscr{A}_{d,m}(Z)$, choose a parameter $t$ of $\A^1$ and consider the following polynomials in $\kk[\A^1] = \kk[t]$:

\begin{equation}\label{eq:basic-polynomial}
M(t)  = \prod_{i=1}^m(t+a_i), \ \  V(t)  = \sum_{j=0}^{\frac{d-m}{p}} b_{j}^{1/p} t^j, \ \ W(t) = \sum_{j=0}^{\frac{d-m}{p}} c_{j}^{1/p} t^j,
\end{equation}
for arbitrary $a_i, b_{j}^{1/p}, c_{j}^{1/p} \in \kk$. When $m=0$, we set $M(t) = 1$.

We define an $\A^1$-curve $\interior{f} \colon \A^1 \to \interior{Z}$ using  \eqref{eq:basic-polynomial} such that
\begin{equation}\label{eq:parameterizing-A1}
(\interior{f})^*U_i = u_i(t) := M(t) \cdot V(t)^{p-i} \cdot W(t)^i, \ \mbox{for any} \ i = 0, 1, \cdots, p.
\end{equation}
It is straight forward to check that $u_i(t)$ satisfy all the relations in \eqref{eq:binomial-relation}, hence a well-defined  $\interior{f}$. Indeed, these are all that we have:

\begin{thm}\label{thm:parameterize-A1}
Every $\A^1$-curve in $\mathscr{A}_{d,m}(Z)$ has a parameterization of the form \eqref{eq:parameterizing-A1} for an appropriate choice of $a_i, b_{j}^{1/p}, c_{j}^{1/p} \in \kk$ in \eqref{eq:basic-polynomial}.
\end{thm}

\begin{proof}
Consider an $\A^1$ curve $\interior{f} \colon \A^1 \to \interior{Z}$. We will construct a parameterization of $\interior{f}$ in the form \eqref{eq:parameterizing-A1} as follows. 

\smallskip

\noindent
{\bf Step 1: The $m=0$ case.} By Proposition \ref{prop:lift-to-Y}, $\interior{f}$ has a unique lift $\interior{f}_Y \colon \A^1 \to \interior{Y}$ such that $\interior{f} = F_{Y/Z} \circ \interior{f}_Y$. Since $\interior{Y} \cong \A^2$ has the coordinate ring $\kk[\frac{y_0}{y_2},\frac{y_1}{y_2}]$, the morphism $\interior{f}_Y$ is of the form:
\[
(\interior{f}_Y)^*(\frac{y_0}{y_2}) = v(t) \ \ \mbox{and} \ \ (\interior{f}_Y)^*(\frac{y_1}{y_2}) = w(t),
\]
where $v(t), w(t) \in \kk[t]$. 
Using the coordinate description of $F_{Y/Z}$ in \eqref{eq:Y-Z-foliation}, we compute that:
\[
(\interior{f})^*U_i = (F_{Y/Z} \circ \interior{f}_Y)^*U_i = (\interior{f}_Y)^*\Big((\frac{y_0}{y_2})^{p-i}\cdot(\frac{y_1}{y_2})^{i}\Big) = v(t)^{p-i} \cdot w(t)^{i}. 
\]
Thus $\interior{f}$ is of the form \eqref{eq:parameterizing-A1} by setting $M(t) = 1, V(t) = v(t)$, and $W(t) = w(t)$. 

\smallskip
{\bf Step 2: The $m > 0$ case.} Suppose $\interior{f}$ passes through $\str_Z$ at $t = -a_1, \cdots, -a_m$ for $a_i \in \kk$. Since $\str_Z$ is the point defined by $U_i = 0$ for all $i$, the morphism $\interior{f}$ is of the form:
\[
(\interior{f})^*U_i = M(t) \cdot \tilde{u}_i(t), \ \ \mbox{for all} \ i = 0, 1, \cdots, p.
\]
for $M(t)$ a polynomial of the form in \eqref{eq:basic-polynomial}, and $\tilde{u}_i(t) \in \kk[t]$ satisfying $\tilde{u}_i(0) \neq 0$ for all $i$. Using  \eqref{eq:binomial-relation}, we deduce that the polynomials  $\tilde{u}_i$ satisfy similar relations:
\[
\tilde{u}_i(t)\cdot \tilde{u}_j(t) = \tilde{u}_{i'}(t)\cdot \tilde{u}_{j'}(t), \ \ \mbox{for} \ i+j = i' + j'.
\]
Thus we obtain another $\A^1$-curve $\interior{g} \colon \A^1 \to \interior{Z}$ by setting
\[
(\interior{g})^*U_i = \tilde{u}_i, \ \ \mbox{for all} \ i = 0, 1, \cdots, p.
\]
Note that $\interior{g}$ avoids the point $\str_Z$. By Step 1, we see that $\tilde{u}_i$ is of the form
\[
\tilde{u}_i = V(t)^{p-i} \cdot W(t)^i, \ \ \mbox{for any} \ i = 0, 1, \cdots, p.
\]
for $V(t)$ and $W(t)$ as in \eqref{eq:basic-polynomial}. 

\smallskip

This finishes the proof. 
\end{proof}

We now summarize the properties of the moduli of $\A^1$-curves in $Z$. 

\begin{corollary}\label{cor:d-m-curve-in-Z}
The component $\mathscr{A}_{d,m}(Z)$ admits a surjective morphism from an affine space: 
\begin{equation}\label{eq:affine-space-to-AZ}
\A^{\dim \mathscr{A}_{d,m}(Z) + 1} \longrightarrow \mathscr{A}_{d,m}(Z).
\end{equation}
In particular, it is irreducible of dimension $\frac{2}{p} \cdot d + (1-\frac{2}{p}) \cdot  m$. 
\end{corollary}

\begin{remark}
	The morphism  \eqref{eq:affine-space-to-AZ} is inseparable, since in the parameterization \eqref{eq:parameterizing-A1} , the factorizations \eqref{eq:basic-polynomial} depend on taking $p$th roots of $u_0$ and $u_p$.  
\end{remark}

\begin{proof}
The dimension of $\mathscr{A}_{d,m}(Z)$ is computed in \eqref{eq:d-m-Z-dim}. It remains to construct the morphism from an affine space as in the statement. 

Note that the parameterization \eqref{eq:parameterizing-A1} relies on $2\cdot \frac{d-m}{p} + m + 2 = \mathscr{A}_{d,m}(Z) + 2$ independent parameters: 
\[
a_1, \cdots, a_m, b_{0}^{1/p}, \cdots, b_{\frac{d-m}{p}}^{1/p}, c_{0}^{1/p}, \cdots, c_{\frac{d-m}{p}}^{1/p}
\]
By choosing the coordinate $t$ of $\A^1$ carefully, we may assume $a_m = 0$. Thus, we obtain a morphism \eqref{eq:affine-space-to-AZ} using the rest  $\mathscr{A}_{d,m}(Z) + 1$ parameters via \eqref{eq:parameterizing-A1}. The surjectivity follows from Theorem \ref{thm:parameterize-A1}. 
\end{proof}

\begin{corollary}\label{cor:AZ-connected}
$\mathscr{A}_{d}(Z)$ is connected with $\lfloor d/p \rfloor + 1$ irreducible components given by $\mathscr{A}_{d,m}(Z)$ for integers $0\leq m \leq d$ such that $m \equiv d \mod p$.
\end{corollary}
\begin{proof}
The parameterization \eqref{eq:parameterizing-A1} and \eqref{eq:basic-polynomial} shows that every $\A^1$-curve in $Z$ can be degenerate to an $\A^1$-curve which is a multiple cover of an $\A^1$-line. In particular, the intersection $\mathscr{A}_{d,m}(Z) \cap \mathscr{A}_{d,d}(Z)$ is non-empty. Thus, the statement follows from Corollary \ref{cor:d-m-curve-in-Z}.
\end{proof}

\subsection{Descending back to $X$}\label{ss:back-to-X}

Induced by \eqref{eq:parameterizing-A1}, we obtain parameterizations of all $\A^1$-curves in $X$. 

\begin{corollary}\label{cor:parameterizing-A1-X}
Let $\interior{g} \colon \A^1 \to \interior{X}$ be the composition of $\interior{f}$ as in \eqref{eq:parameterizing-A1} with the morphism $Z \to X$. Then in homogeneous coordinates, we have 
\begin{equation}\label{eq:parameterizing-A1-X}
\begin{split}
(\interior{g})^*(x_0) =& M(t) \cdot V^{p} \\
(\interior{g})^*(x_1) =& M(t) \cdot W^p\\
(\interior{g})^*(x_2) =& M(t) \cdot \sigma^{1/p}(V,W) - 1.
\end{split}
\end{equation}
\end{corollary}
\begin{proof}
Note that the right hand side of \eqref{eq:parameterizing-A1-X} is indeed given by 
\[
\frac{(\interior{g})^*(x_0)}{z_2}, \ \  \frac{(\interior{g})^*(x_1)}{z_2}, \ \  \frac{(\interior{g})^*(x_2)}{z_2}. 
\]
 Now let $t_0$ be a root of one of $M(t), V(t), W(t)$. Then $(\interior{g})^*(x_2)(t_0) \neq 0$. This shows that \eqref{eq:parameterizing-A1-X} defines a morphism as needed.
\end{proof}

\begin{proof}[Proof of Theorem \ref{thm:A1-curve-X}]
Composing along the sequence 
\[
H \to Z \to X,
\]
we conclude that 
\begin{enumerate}
 \item[(i)] The connectedness and the classification of the irreducible components of $\mathscr{A}_{pd}(X)$ follow from Corollary \ref{cor:AZ-connected}. 
 \item[(ii)] Statement (2) follows from \eqref{eq:d-m-Z-dim}. 
 \item[(iii)] Statement (3) follows from Corollary \ref{cor:d-m-curve-in-Z}. 
\end{enumerate}
This finishes the proof.
\end{proof}

\begin{proof}[Proof of Proposition \ref{intro-prop:A1-connectedness}]
The inseparability and the $p=2$ case in the statement follows from Proposition \ref{prop:SAC-failure}. 

For $\A^1$-connectedness in (1), Lemma \ref{lem:lift-in-H} (3) implies that $H$ is separably $\A^1$-connected by curves parameterized by $\sA_{d,m}(H)$ for $0<m < d$. The tautological morphism 
$
\sA_{d,m}(H) \to \sA_{d,m}(X)
$
induced by the surjection $H \to X$ 
imply the $\A^1$-connectedness of $X$ by $\A^1$-curves in $\sA_{d,m}(X)$ for $0<m < d$. 

In the $m=0$ case, the $\A^1$-connectedness of $X$ follows from Proposition \ref{prop:lift-to-Y}, noting that any $\A^1$-curve in $Y$ is very free \cite[Theorem 1.10]{CZ19A1}. 
\end{proof}


\section{The component $\sA_{pd,d-p}(X)$}
The largest multiplicity $m$ such that $\sA_{pd,m}(X) \neq \emptyset$ is $m = d$, which parameterizes covers of $\A^1$-lines in $X$. The next largest multiplicity is $m = d-p$ which is the topic of this section. 

\subsection{The defining equation}
Given $m = d-p$, the parameterization \eqref{eq:basic-polynomial} is of the form 
\begin{equation}\label{eq:d-p-parameterization}
u_{i}(t) = (b_{0}^{1/p} + b^{1/p}_1 t)^{p-i} \cdot (c_{0}^{1/p} + c^{1/p}_{1}t)^i \cdot  \prod_{j=1}^m(t+a_j) 
\end{equation}
for $i = 0, 1, \cdots, p$, where
\begin{equation}\label{eq:d-p-MVW}
M(t) = \prod_{j=1}^m(t+a_j), \ \ V(t) =  b_{0}^{1/p} + b^{1/p}_1 t, \ \ W(t) = c_{0}^{1/p} + c^{1/p}_{1}t. 
\end{equation}
We make the following notations
\begin{equation}\label{eq:L-pi}
\begin{split}
L_k &:= c_k x_0 - b_k x_1, \ \ \mbox{for} \ \ k = 0, 1, \\
\pi &:= b_1 c_0 - b_0 c_1. 
\end{split}
\end{equation}

\begin{proposition}\label{prop:d-p-equation}
Consider an $\A^1$-curve in $\sA_{pd,d-p}(X)$ with the parameterization \eqref{eq:d-p-parameterization}. Then its image in $X$ is given by
 \begin{equation}\label{eq:d-p-equation}
 	L_1^d-(-1)^d \pi^p\Delta_X\prod_{j=1}^{d-p}(L_0-a_j^p L_1) = 0,
 \end{equation}
 where we use $\Delta_X$ for the equation $\sigma(x_0,x_1)-x_2^p$. Note that when $d=p$, the product in the above equation is $1$. 
\end{proposition}

\begin{proof}
Let $g \colon \P^1_{\infty} \to X$ be the $\A^1$-curve and $f: \P^1_\infty \to Z$ the lift of $g$ with parameterization \eqref{eq:d-p-parameterization}. To simplify the notations, we identify $x_i$ with its image in the coordinate ring of $Z$ via $Z \to X$. Note that 
\[
\begin{split}
u_0 &= (b_0 + b_1 t^p)\cdot \prod_{j=1}^m(t+a_j) \\
u_p &= (c_0 + c_1 t^p)\cdot \prod_{j=1}^m(t+a_j),
\end{split}
\]
where $m = d -p$. Note, if $m=0$ the product $ \prod_{j=1}^m(t+a_j) = 1$. We compute:
\begin{align*}
	  f^*\left( \frac{L_0 - a_i^p L_1}{z_2} \right)  &= (c_0 u_0 - b_0 u_p) - a_i^p (c_1 u_0 - b_1 u_p) \\
& = \prod_{j=1}^m(t+a_j)  \cdot \Big(  (c_0  - a^p_i c_1 )(b_0 + b_1 t^p) + (a^p_i b_1 - b_0)(c_0 + c_1 t^p) \Big) \\
 &= \pi \prod_{j=1}^m(t+a_j)  \cdot  ( a_i +  t  )^p.
\end{align*}
So taking the product over all of these $d-p$ factors we have:
\[
\prod_{j=1}^{d-p}f^*\left( \frac{(L_0-a_j^p L_1)}{z_2} \right) = \pi^{d-p} \cdot \prod_{j=1}^m(t+a_j)^d.
\]
Also observe that 
\[
f^*\left( \frac{\Delta_X}{z^p_2} \right) = 1.
\]
Combining the above calculation, we obtain 
\begin{equation}\label{eq:term-2}
f^*\left( \frac{\pi^p\Delta_X \cdot  \prod_{j=1}^{d-p}(L_0-a_j^p L_1)}{z_2^d} \right) = \pi^d \cdot \prod_{j=1}^m(t+a_j)^d.
\end{equation}

On the other hand, we have 
\begin{equation}\label{eq:term-1}
\begin{split}
 & f^*\left( \frac{L_1}{z_2} \right)^d  = \left( c_1u_0 - b_1 u_p \right)^d \\
 =& \prod_{j=1}^m(t+a_j)^d  \cdot \Big( c_1 (b_0 + b_1 t^p) - b_1(c_0 + c_1 t^p)\Big)^d   \\
=& (-1)^d \pi^d \cdot \prod_{j=1}^m(t+a_j)^d.
\end{split}
\end{equation}
The statement follows from combining \eqref{eq:term-2} and \eqref{eq:term-1}.
\end{proof}

Choosing the coordinate $t$ appropriately, we may assume that $a_m = 0$ in \eqref{eq:d-p-parameterization}.  Set
\begin{equation}\label{eq:d-p-equation-simple}
\Psi = 
\begin{cases}
L_1^d-(-1)^d \pi^p\Delta_X, \ &\mbox{for} \ d - p =0; \\
L_1^d-(-1)^d \pi^p\Delta_X L_0 , \ &\mbox{for} \ d - p =1; \\
L_1^d-(-1)^d \pi^p\Delta_X \cdot L_0  \displaystyle \prod_{j=1}^{d-p-1}(L_0-a_j^p L_1), \ &\mbox{for} \ d-p > 1.
\end{cases}
\end{equation}
Thus, we obtain
\begin{corollary}\label{cor:d-p-equation}
Every $\A^1$-curve in $\sA_{pd,d-p}(X)$ has its image cut out by precisely $\Psi = 0$. 
\end{corollary}

\subsection{Singularities of a general $\A^1$-curve}
Consider the sequence of dominant morphisms
\[
\mathscr{A}_{d,m}(H) \longrightarrow \mathscr{A}_{d,d-p}(Z) \longrightarrow \mathscr{A}_{pd, d-p}(X).
\]
Let $\tilde{f} \colon \P^1_{\infty} \to H$ be any $\A^1$-curve in $\mathscr{A}_{d,m}(H)$, and $f: \P^1_{\infty} \to X$ obtained by composing $\tilde{f}$ with $H \to Z \to X$. 

\begin{proposition}\label{prop:general-fiber-singularity}
Notations as above, $f$ is birational onto its image  $f(\P^1_{\infty})$, which away from $\str$, is singular for $d > 2$ with only cusps. 
Furthermore for a general $f$,  $\str$ is an ordinary point of $f(\P^1_{\infty})$ with multiplicity $m$. 
\end{proposition}
\begin{proof}
By Lemma \ref{lem:lift-in-H} (2), $\tilde{f}$ is a section of the projection $H \to E_Z$, hence is a smooth $\A^1$-curve in $H$. Thus $f$ is either birational onto its image, or $f$ is a degree $p$ inseparable cover onto its image. In the latter case, we have $p|d$ and the equation $\Psi$ corresponding to $f$ is of the form $\Psi = \psi^p$ for some degree $d/p$ homogeneous polynomial $\psi \in \kk[x_0, x_1, x_2]$. We show that this isn't possible. 

In the $d-p > 1$ case, $\Psi = \psi^{p}$ implies  that
\[
(-1)^d \pi^p\Delta_X \cdot L_0  \prod_{j=1}^{d-p-1}(L_0-a_j^p L_1) = L^d_1 - \psi^p = (L_1^{d/p} - \psi)^p.
\] 
Since $\Delta_X$ is irreducible and coprime to $L_0  \prod_{j=1}^{d-p-1}(L_0-a_j^p L_1)$, both $\Delta_X$ and $L_0  \prod_{j=1}^{d-p-1}(L_0-a_j^p L_1)$ have $p$-th root unless both sides of the above equation are zero. In the latter case, $f$ is a degree $d$ cover of the $\A^1$-line $L_1 = 0$, whose lift $\tilde{f}$ in $H$ is not contained in $\mathscr{A}_{d,m}(H)$ by Lemma \ref{lem:lift-in-H} (2). This is a contradiction! The same proof applies to the $d-p =0$ and $d-p = 1$ cases as well. This proves that $f$ is birational onto its image. 

As for the singularities of $f(\P^1_{\infty})$, note that the morphism $H \to X$ is purely inseparable, except along the exceptional curve $E_{Z}$ which contracts to $\str$. Thus away from $\str$, $f$ can only have uni-branched singularities or cusps. 

To show the existence of cusps, note that the delta-invariant of the ordinary point $\str$ is $\frac{m(m-1)}{2}$. By the genus-degree consideration,  $f(\P^1_{\infty})$ has singularities other than $\str$ to contribute the delta-invariant
\[
\frac{(d-1)(d-2)}{2} - \frac{m(m-1)}{2} = \frac{(p-1)(2d-p-2)}{2}.
\]
This integer is positive unless $d=p=2$. Thus the image $f(\P^1_{\infty})$ has at least one cusp for $d > 2$.

Finally, since the choice of $a_i$ are general, $\str$ is ordinary with multiplicity $m$ for a general $f(\P^1_{\infty})$ by the parameterization \eqref{eq:d-p-parameterization}. 
\end{proof}

\subsection{Smoothness around cusps}

Recall that $m= d-p$. For $m> 1$, we fix a general choice $a_1^p, \cdots a_{m-1}^p \in \kk$, and consider the hypersurface
\[
\mathcal{C} \subset \P^2\times\A^4
\]
cut out by $\Psi = 0$. Here $\P^2 = \uX$ and $\A^4$ has its coordinates given by $b_0, b_1, c_0, c_1$. The closed fibers of the projection
\begin{equation}\label{eq:d-p-family}
\mathcal{C} \to \A^4
\end{equation}
are $\A^1$-curve by Corollary \ref{cor:d-p-equation}. Furthermore for $d > 2$, a general fiber of this projection has cusps by Proposition \ref{prop:general-fiber-singularity}. However, we next show that these cusps on the fiber are smooth points of the total space $\mathcal{C}$. 

Set 
\[
\mathcal{L} := 
\begin{cases}
L_0^m,   \ \ &\mbox{if} \ m =0, 1;\\
 L_0  \displaystyle \prod_{j=1}^{m-1}(L_0-a_j^p L_1), \ \ &\mbox{if} \ m > 1.
\end{cases}
\]
Then the defining equation $\Psi$ of $\mathcal{C}$ takes the form 
\[
\Psi = L_1^d - (-1)^d \pi^p \Delta_X \mathcal{L}. 
\]

Recall the elementary symmetric polynomials of degree $k$: 
\[
e_k(A_1, \cdots, A_{m-1}) = \sum_{1\leq j_1 < j_2 < \cdots < j_k \leq m-1} A_{j_1} A_{j_2}\cdots A_{j_k}.
\]
Noting that $e_k$ is homogeneous, we compute 
\[
\begin{split}
\prod_{j=1}^{m-1}(L_0-a_j^p L_1) &= L_0^{m-1} + \sum_{k=1}^{m-1}e_{k}(-a_1^p L_1, \cdots, -a_{m-1}^p L_1)L_{0}^{m-1-k} \\
&= L_0^{m-1} + \sum_{k=1}^{m-1}e_{k}L_0^{m-1-k}L_1^{k}
\end{split}
\]
where we use $e_{k} := e_{k}(-a_1^p, \cdots, -a_{m-1}^p)$. Thus, we have
\[
\mathcal{L} = L_0^m + \sum_{k=1}^{m-1} e_k L_0^{m-k} L_1^k
\]
where if $m < 1$ the sum is empty hence is zero. 

For convenience, we introduce the two notations:
\[
\mathcal{D} :=
\begin{cases}
d L_1^{d-1}, \ &\mbox{for} \ m=0,1; \\
d L_1^{d-1} - (-1)^d \pi^p \Delta_X \displaystyle \sum_{k=1}^{m-1}ke_kL_0^{m-k}L_1^{k-1}, \ &\mbox{for} \ m > 1.
\end{cases} 
\]
and
\[
\mathcal{E} := 
\begin{cases}
 m, &\mbox{for} \ m=0,1; \\
 mL_0^{m-1} + \displaystyle \sum_{k=1}^{m-1}(m-k)e_kL_0^{m-k-1}L_1^k, \  &\mbox{for} \ m> 1. 
\end{cases}
\]
We compute the gradient:

\begin{lemma}\label{lem:gradient-Psi}
\[
\begin{split}
\nabla \Psi &= \left(\partial_{x_0} \Psi, \partial_{x_1} \Psi, \partial_{x_2}, \partial_{b_0} \Psi, \partial_{b_1}\Psi , \partial_{c_0} \Psi, \partial_{c_1} \Psi  \right) \\
&= \mathcal{D} \cdot \nabla L_1 - (-1)^d\pi^p \mathcal{L} \cdot\nabla \Delta_X - (-1)^d \pi^p \mathcal{E} \Delta_X \cdot \nabla L_0.
\end{split}
\]
where 
\[
\begin{split}
\nabla L_1 &= (c_1, -b_1, 0, 0, -x_1, 0, x_0), \\
\nabla \Delta_X &= (\partial_{x_0} \Delta_X, \partial_{x_1}\Delta_X, 0,0,0,0,0),\\
\nabla L_0 &= (c_0, -b_0, 0, -x_1, 0, x_0, 0). 
\end{split}
\]

In particular, when $m=0$ we have 
\[
\nabla \Psi = - (-1)^d\pi^p \mathcal{L} \cdot\nabla \Delta_X.
\]
\end{lemma}
\begin{proof}
In case $m=0$ or equivalently $d=p$, we have $\mathcal{D} = \mathcal{E} = 0$. Hence the second statement follows from the first one. 
The first statement can be verified by a straight forward computation. We record the calculations below for $m > 1$. The cases of $m=0,1$ are similar and  easier, hence is left to the reader. 

\begingroup\leqnos
\begin{align*}
\tag{1} 
\partial_{x_0}\Psi 
 &=
  d c_1 L_1^{d-1} - (-1)^{d}\pi^p\Big(\partial_{x_0}\Delta_X \cdot \mathcal{L} + \Delta_X \cdot \partial_{x_0}\mathcal{L} \Big) 
  \\
&=   d c_1 L_1^{d-1}
 - (-1)^{d}\pi^p \partial_{x_0}\Delta_X \cdot \mathcal{L} 
 %
 - (-1)^{d}\pi^p \Delta_X
 \bigg[ 
 m\cdot c_0 L_0^{m-1} 
 \\&
 \quad 
  +\sum_{k=1}^{m-1}\bigg( (m-k)e_k c_0 L_{0}^{m-k-1}L_1^k 
 %
+ 
ke_k c_1 L_0^{m-k} L_1^{k-1}\bigg)
 \bigg]
 \\
&= c_1 \mathcal{D} 
- (-1)^d\pi^p\mathcal{L} \cdot \partial_{x_0}\Delta_X  
- (-1)^d\pi^p\Delta_X c_0 \mathcal{E}. 
\\
%
%
%
%
%
%
%
%
%
%
%
%
%
%
%
\tag{2}
\partial_{x_1}\Psi
& =
 d b_1 L_1^{d-1} 
- (-1)^{d}\pi^p\Big(\partial_{x_1}\Delta_X \cdot \mathcal{L} 
+ \Delta_X \cdot \partial_{x_1}\mathcal{L} \Big)
 \\
&=   d b_1 L_1^{d-1} 
- (-1)^{d}\pi^p \partial_{x_1}\Delta_X \cdot \mathcal{L} 
\\&\quad
- (-1)^{d}\pi^p \Delta_X
\bigg[ m\cdot (-b_0) L_0^{m-1} 
\\&\quad  
+ \sum_{k=1}^{m-1} \bigg(
-b_0(m-k)e_k  L_{0}^{m-k-1}L_1^k 
-b_1 
ke_k  L_0^{m-k} L_1^{k-1}
\bigg)
\bigg] 
\\
&= -b_1 \mathcal{D} - (-1)^d\pi^p\mathcal{L} \cdot \partial_{x_1}\Delta_X  - (-1)^d\pi^p\Delta_X (-b_0) \mathcal{E}. 
%
%
%
%
%
%
%
%
%
%
%
%
%
\\
%
%
%
%
%
%
%
%
%
%
%
%
%
%
\tag{3}\partial_{x_2}\Psi &= 0.
%
%
%
%
%
%
%
%
%
%
%
%
%
\\
%
%
%
%
%
%
%
%
%
%
%
%
%
%
\tag{4}
\partial_{b_0}\Psi 
&
= -(-1)^d\pi^p\Delta_X \partial_{b_0}\mathcal{L} 
= -(-1)^d\pi^p\Delta_X (-x_1) \mathcal{E}. 
%
%
%
%
%
%
%
%
%
%
%
%
%
\\
%
%
%
%
%
%
%
%
%
%
%
%
%
%
\tag{5}
  \partial_{b_1}\Psi 
&= d(-x_1)L_{1}^{d-1} - (-1)^d \pi^p \Delta_X \partial_{b_1}\mathcal{L} 
\\
&= d(-x_1)L_{1}^{d-1} - (-1)^d \pi^p \Delta_X \left( \sum_{k=1}^{m-1}ke_k(-x_1)L_0^{m-k}L_1^{k-1} \right) 
\\
&= (-x_1) \mathcal{D}. 
%
%
%
%
%
%
%
%
%
%
%
%
%
\\
%
%
%
%
%
%
%
%
%
%
%
%
%
%
\tag{6}
\partial_{c_0}\Psi 
&
= -(-1)^d \pi^p \Delta_X \partial_{c_0}\mathcal{L} 
= -(-1)^d\pi^px_0 \Delta_X \mathcal{E}.
%
%
%
%
%
%
%
%
%
%
%
%
%
\\
%
%
%
%
%
%
%
%
%
%
%
%
%
%
\tag{7}
\partial{c_1}\Psi 
&
= dx_0L_1^{d-1} - (-1)^d\pi^p \Delta_X \partial_{c_1}\mathcal{L} 
= x_0 \mathcal{D}. 
\end{align*}
\endgroup 
\end{proof}

\begin{lemma}\label{lem:C-singularity}
The total space $\mathcal{C}$ is singular along the following loci:
\begin{enumerate}
\item The points whose images in $\uX = \P^2$ are the singularities of $\Delta_{X}$. 
\item The points cut out by $\pi = 0$.
\item The points whose images in $\uX = \P^2$ are $\str$ when $ m > 1$. 
\end{enumerate}
Furthermore, when $m = 0, 1$, the points in (3) are smooth as long as they are not in (2). 
\end{lemma}
\begin{proof}
For a point $s \in \mathcal{C}$ in (1), note that  $\nabla \Delta_X(s) = 0$ and $\Delta_X(s) = 0$. Together with $\Psi(s) = 0$, this implies that $L_1(s) = 0$ hence $\mathcal{D} = 0$. Using Lemma \ref{lem:gradient-Psi}, we obtain that $\nabla{\Psi}(s) = 0$. Hence $s$ is a singular point of $\mathcal{C}$. 

For a point $s \in \mathcal{C}$ in (2), the condition $\Psi(s) = \pi(s) = 0$ implies that $L_1(s)=0$ hence $\mathcal{D} = 0$. Lemma \ref{lem:gradient-Psi} again implies $\nabla{\Psi}(s) = 0$. Hence $s$ is a singular point of $\mathcal{C}$.

Consider a point $s \in \mathcal{C}$ mapping to $\str$ in $\uX$. Note that $\Delta_X(s) = -1$ and $L_0(s) = L_1(s) = 0$. 

For $m> 1$, $\mathcal{E}(s) = 0$ implies that
\begin{equation}
\nabla \Psi = (-1)^d \pi^p \mathcal{E} \cdot  (c_0, -b_0, 0, 0, 0, 0, 0) = 0.
\end{equation}
This gives the singularity in (3). 

For $m=1$, we have $\mathcal{E} = 1$, hence 
\begin{equation}
\nabla \Psi = (-1)^d \pi^p \cdot  (c_0, -b_0, 0, 0, 0, 0, 0) \neq 0,
\end{equation}
as long as $\pi^p \neq 0$, which is smooth. 

For $m=0$, the condition $s \in \mathcal{C}$ implies that $\pi^p(s) = 0$ by the definition of $\Psi$. 
\end{proof}
 
Note that $\Psi = L_1^d$ over the locus $(\pi=0) \subset \A^4$, which corresponds to covers of the $\A^1$-line $(L_1 = 0) \subset X$. Denote by $\mathcal{U} \subset \A^4 \setminus (\pi = 0)$ such that the fibers over the pull-back family
\begin{equation}\label{eq:smooth-todal-space}
\mathcal{C}_{\mathcal{U}} := \mathcal{C}\times_{\A^4}\mathcal{U} \longrightarrow \mathcal{U}.
\end{equation}
avoids the singularity of $\Delta_X$. Noting that over $\mathcal{C}$
\[
\xymatrix{
(\Psi = \Delta_X = 0) \ar@{<=>}[r] & (L_1 = \Delta_X =0),
}
\]
the locus $\mathcal{U}$ is open and dense in $\A^4$. Recall by Proposition \ref{prop:general-fiber-singularity},  a general fiber of \eqref{eq:smooth-todal-space} has cusp(s) away from $\str$ when $d > 2$. However, we will show that these cusp(s) needs not to be a singularity of the total space $\mathcal{C}_{\mathcal{U}}$, and most of the time, they don't! 

\begin{lemma}\label{lem:CU-singularity-criterion}
Let $s$ be a singularity of $\mathcal{C}_{\mathcal{U}}$ whose image in $\uX$ is not $\str$. Then we have 
$
\mathcal{D}(s) = \mathcal{E}(s) = 0.
$
\end{lemma}
\begin{proof}
Let $s$ be a singular point of $\mathcal{C}_{\mathcal{U}}$, hence 
\[
\nabla \Psi(s) = 0 \ \ \mbox{and} \ \ \Psi(s) = 0. 
\]
Consider the last $4$ entries in $\nabla \Psi$. Since $\pi(s) \neq 0$, by $\nabla \Psi(s) = 0$ and Lemma \ref{lem:gradient-Psi}, we have 
\[
\mathcal{D}(s)x_1(s) = \mathcal{D}(s)x_0(s) = \mathcal{E}(s)x_1(s) = \mathcal{E}(s)x_0(s) = 0.
\]
By assumption on the image of $s$, we have $x_0(s) \neq 0$ or $x_1(s) \neq 0$. This implies the statement. 
\end{proof}

\begin{proposition}\label{prop:sing-total-space}
\begin{enumerate}
\item If $m=0$, then the number of singularities of $\mathcal{C}_{\mathcal{U}}$ along a general fiber of $\mathcal{C}_{\mathcal{U}} \to \mathcal{U}$ is equal to the number of singularities of $\Delta_X$, which is a positive integer $\leq p-2$. 
\item If $m=1$, then $\mathcal{C}_{\mathcal{U}}$ is smooth. 
\item If $m > 1$, then a general fiber of $\mathcal{C}_{\mathcal{U}} \to \mathcal{U}$ contain no singularities of $\mathcal{C}_{\mathcal{U}}$ other than $\str$. Furthermore, if $p=2$ then $\mathcal{C}_{\mathcal{U}}$ is smooth away from $\str$.
\end{enumerate}

\end{proposition}
\begin{proof}
Noting that $\mathcal{E} = 1$ when $m=1$. Hence the case $m=1$ in (2) follows from Lemma \ref{lem:C-singularity} and \ref{lem:CU-singularity-criterion}.

Next we assume that $m=0$ or $m>1$. Let $s$ be a singularity of $\mathcal{C}_{\mathcal{U}}$ not mapping to $\str$, i.e. at least one of $x_0(s) \neq 0$ and $x_1(s)\neq 0$ holds. 

By Lemma \ref{lem:gradient-Psi} and \ref{lem:CU-singularity-criterion}, the condition $\nabla\Psi(s) = 0 $ implies that 
\begin{equation}\label{eq:psi-partial}
\mathcal{L}(s) \partial_{x_0}\Delta_X(s) = \mathcal{L}(s) \partial_{x_1}\Delta_X(s) = 0. 
\end{equation}
hence either $\mathcal{L}(s) = 0$ or $\partial_{x_0}\Delta_X(s) = \partial_{x_1}\Delta_X(s) = 0$. 

\bigskip

\noindent
{\bf Case 1: $\mathcal{L}(s) = 0$.}
Note that $\mathcal{L} = 1$ when $m=0$. So we only need to consider $m>1$ in this case. The condition $\Psi(s) = 0$ implies $L_1(s) = 0$. By the definitions of $\mathcal{D}$, we have 
\[
0 = \mathcal{D}(s) =
 - (-1)^d \pi(s)^p \Delta_X(s) e_1 L_0^{m-1}(s).
\]
But $e_1 \neq 0$ for a general choice of $a_1^p,\cdots, a_{m-1}^p$. Hence we obtain $L_0(s) = 0$. This implies that the image of $s$ in $\uX$ is $\str = (L_0=0) \cap (L_1=0)$, contradicting our assumption. 

\bigskip

\noindent
{\bf Case 2: $\mathcal{L}(s) \neq 0$.} 
In this case we have $\partial_{x_0}\Delta_X(s) = \partial_{x_1}\Delta_X(s) = 0$ by \eqref{eq:psi-partial}. Note that
\begin{equation}\label{eq:Delta-partial}
\begin{split}
\partial_{x_0}\Delta_X &= \sum_{i=1}^{p-1}i\sigma_i x^{i-1}_0 x_1^{p-i} = x_1 \cdot P(x_0, x_1), \\
\partial_{x_1}\Delta_X &= -\sum_{i=1}^{p-1}i\sigma_i x^{i}_0 x_1^{p-i-1} = -x_0 \cdot P(x_0, x_1).
\end{split}
\end{equation}
where $P(x_0, x_1) = \sum_{i=1}^{p-1}i\sigma_i x^{i-1}_0 x_1^{p-i-1}$ is a homogeneous polynomial of degree $p-2$. In case $p=2$, $P$ is a non-zero constant, hence $s$ maps to $\str$. This proves the statement when $p=2$  in $(1)$ and $(3)$. 

Suppose $p>2$. Take the linear factorization
\[
P(x_0, x_1) = P_1 \cdots P_{p-2},
\]
where $P_i$'s are not necessarily distinct. The set $\{(P_i = 0) \subset \uX\}$ are precisely the set of lines joining $\str$ with a singularity of $\Delta_X$. Furthermore, each line $(P_i = 0)$ lifts to a fiber of the projection $H \to E_Z$. Thus for a general $w \in \mathcal{U}$ the intersection $(P_i = 0)\cap \mathcal{C}_w$ is supported at a single point, denoted by $w_i$. This implies that if $s \in \mathcal{C}_w$ then necessarily $s \in \{w_i\}_{i=1}^{p-2}$. 

\bigskip
\noindent
{\bf Subcase (a): $m=0$.} By Case 1 and the second statement in Lemma \ref{lem:gradient-Psi}, $\{w_i\}_{i=1}^{p-2}$ are precisely the singularities of $\mathcal{C}_{\mathcal{U}}$ along a general fiber $\mathcal{C}_w$. This proves (1). 

\bigskip
\noindent
{\bf Subcase (b): $m>1$.} By Lemma \ref{lem:CU-singularity-criterion}, $w_i \neq \str$ is a singularity of $\mathcal{C}_{\mathcal{U}}$ only if
\[
\mathcal{E}(w_i)=0.
\]
Note that $\mathcal{E}$ is a non-trivial degree $m-1$ homogeneous polynomial in $\kk[L_0, L_1]$. Take the linear factorization
\[
\mathcal{E} = \prod_{k=1}^{m-1} (A_k L_0 + B_k L_1)
\]
where coefficients $A_k, B_k$ are determined by $\{e_j\}_{j=1}^{m-1}$. Thus for a general $\mathcal{C}_w$, we can make sure that the lines $(P_i = 0)$ for $i = 1, \cdots, p - 2$ do not appear in the set of lines $\{(A_k L_0 + B_k L_1) = 0\}_{k=1}^{m-1}$. Consequently, $\mathcal{E}(w_i)\neq 0$ unless $w_i = \str$. This proves (3). 
\end{proof}

The following result was known \cite{VA91}. We record it below for convenience.
\begin{corollary}\label{cor:boundary-sing}
A general $\Delta_X$ has $(p-2)$ cusps. 
\end{corollary}
\begin{proof}
Note that singularities of $\Delta_X$ are given by precisely the solutions of the equation $P(x_0,x_1) = 0$ in \eqref{eq:Delta-partial}. The polynomial $P(x_0,x_1)$ is of degree $(p-2)$, and can be made separable for a general choice of $\Delta_X$. 
\end{proof}

\subsection{Counting the number of cusps}
Let $\mathcal{C}_w$ be a fiber of the family $\mathcal{C}_{\mathcal{U}} \to \mathcal{U}$ as in \eqref{eq:smooth-todal-space}.  Our next goal is to compute the number of cusps on a general $\mathcal{C}_w$. 

Denote by $\interior{g}: \A^1 \to \interior{X}$ the parameterization of $\mathcal{C}_w$ given by composing the parameterization \eqref{eq:d-p-parameterization} with $Z \to X$. By Corollary \ref{cor:parameterizing-A1-X}, $\interior{g}$ takes the following form
\[
\begin{split}
x_0(t) := (\interior{g})^*(x_0) =& M(t) \cdot V^{p} \\
x_1(t) := (\interior{g})^*(x_1) =& M(t) \cdot W^p\\
x_2(t) := (\interior{g})^*(x_2) =& M(t) \cdot \sigma^{1/p}(V,W) - 1.
\end{split}
\]
with $M, V, W$ given by \eqref{eq:d-p-MVW}. Denote by $\sigma^{1/p}(t) := \sigma^{1/p}(V,W)$ for simplicity. The tangent line of $\mathcal{C}_w$ at a point $\interior(g)(t) \in \mathcal{C}_w$,  if exists is given by the following equation if exists
\[
\det 
\begin{pmatrix}
x_0 & x_1 & x_2 \\
x_0(t) & x_1(t) & x_2(t) \\
x'_0(t) & x'_1(t) & x'_2(t)
\end{pmatrix}
= (M' + M^2 (\sigma^{1/p})') \cdot (W^p x_0 - V^p x_1) =0
\]

\begin{lemma}\label{lem:sigularity-equation}
Cusps of a general $\mathcal{C}_{w}$ are given precisely by the solutions of 
\begin{equation}\label{eq:singularity-equation}
M' + M^2 (\sigma^{1/p})' = 0.
\end{equation}
In particular, we have the upper bound:
\begin{equation}\label{eq:cusp-bounds}
\mbox{number  of cusps on $\mathcal{C}_{w}$} \leq  2d-p-2. 
\end{equation}
\end{lemma}
\begin{proof}
For a general $\mathcal{C}_w$, since $\str$ is an ordinary point, the tangent line of each local branch at $\str$ exists. 
Thus, cusps of $\mathcal{C}_w$ are defined by precisely the non-existence of tangent lines. To prove the first statement, it suffices to show that the marking $\infty_w \in \mathcal{C}_w$ is a smooth point of $\mathcal{C}_w$. By \eqref{eq:d-p-equation}, $\infty_{w}$ is defined by the equation 
\[
L_1 = \Delta_X = 0. 
\]
By Lemma \eqref{lem:gradient-Psi}, we have 
\begin{equation}\label{eq:gradient-at-infty}
\left(\partial_{x_0} \Psi, \partial_{x_1} \Psi, \partial_{x_2}\Psi \right)(\infty_w) = - (-1)^d\pi^p \mathcal{L}(\infty_w) \cdot\nabla \Delta_X(\infty_w)
\end{equation}
By the construction of $\mathcal{C}_{\mathcal{U}} \to \mathcal{U}$, $\infty_w$ is not a singularity of $\Delta_X$, hence $\nabla \Delta_X(\infty_w) \neq 0$. On the other hand, $\mathcal{L}(\infty_w) = 0$ implies that $L_0(\infty_w) = 0$. Since $L_0$ and $L_1$ define two distinct lines, this would imply that $\infty_w = \str$, a contradiction!  Consequently \eqref{eq:gradient-at-infty} is non-zero, and  $\infty_w$ is a smooth point of $\mathcal{C}_w$. 

Finally the upper bound on the number of cusps follow from 
\begin{equation}\label{eq:sing-equation-degree}
\deg \Big(M' + M^2 (\sigma^{1/p})'\Big) = 2m + p -2 = 2d -p -2 .
\end{equation}
\end{proof}

It turns out not all swapping families of $\A^1$-curves provides super-cusps as shown below.

\begin{proposition}\label{prop:m=0-p-neq-2-sing}
Suppose $m=0$ and $\Delta_X$ is general. The number of cusps of a general fiber $\mathcal{C}_w$ is $(p-2)$.  In particular, these cusps are precisely the singularities of the total space $\mathcal{C}_{\mathcal{U}}$ along $\mathcal{C}_w$. 
\end{proposition}
\begin{proof}
Note that in case $m=0$, Equation \ref{eq:singularity-equation} becomes
\[
(\sigma^{1/p})' = 0. 
\]
For a general choices of $\Delta_X$, the above equation is separable, hence has $\deg (\sigma^{1/p})' = p-2$ many distinct solutions. Thus, the first statement follows from Lemma \ref{lem:sigularity-equation}. 

In view of Proposition \ref{prop:sing-total-space} (1) and Corollary \ref{cor:boundary-sing}, these singularities of $\mathcal{C}_w$ are precisely the singularities of the total space.
\end{proof}

To compute the exact number of cusps for a general fiber in all other cases, we make a special choice of $\sigma$: 
\begin{equation}\label{eq:special-choice}
\sigma_0(A,B) = \sum_{i=1}^{p-1} A^i B^{p-i}.
\end{equation}
Note that $\sigma^{1/p}_0(A,B) = \sigma_0(A,B)$. For simplicity, denote by $\sigma_0(t) := \sigma^{1/p}_0(V,W)$ for the linear functions $V, W$ as in  \eqref{eq:d-p-MVW}. 

\begin{lemma}
We compute the derivative: 
\[
\sigma_0'(t) = -\pi^{1/p}(V-W)^{p-2},
\]
where $\pi=b_1 c_0 - b_0 c_1$ is as in \eqref{eq:L-pi}. 
\end{lemma}
\begin{proof}
Note that
\[
\begin{split}
\sigma_0(A,B) &= A^p + \sum_{i=1}^{p-1}A^iB^{p-i} + B^{p} - (A^p + B^p) \\
&= \frac{A^{p+1} - B^{p+1}}{A-B} - (A^p + B^p). 
\end{split}
\]
Taking derivative, we have 
\[
\sigma_0'(t) = (VW' - V'W)\cdot (V-W)^{p-2} = -\pi^{1/p}(V-W)^{p-2}. 
\]
\end{proof}

Thus under the special choice \eqref{eq:special-choice}, Equation \eqref{eq:singularity-equation} becomes
\begin{equation}\label{eq:special-sing-equation}
M' - \pi^{1/p}M^2 (V-W)^{p-2} = 0.
\end{equation}

We first show that in $p=2$ case, only half of the maximal bound in Lemma \ref{lem:sigularity-equation} is achieved.

\begin{proposition}\label{prop:p=2-fiber-singularity}
Suppose $p = 2$, then a general fiber of  $\mathcal{C}_{\mathcal{U}} \to \mathcal{U}$ has $d-2$ cusps as the only singularity. 
\end{proposition}
Note that when $m=d-p=0$, this coincide with Proposition \ref{prop:m=0-p-neq-2-sing}.
\begin{proof}
The $m=0$ case is the direct consequence of the upper bound in Lemma \ref{lem:sigularity-equation} noting that $d = p = 2$ in this case. Next we consider the case $m > 0$, and let $\mathcal{C}_w$ be a general fiber of $\mathcal{C}_{\mathcal{U}} \to \mathcal{U}$. 

Note that $\Delta_X$ is a smooth conic in $p=2$. Thus $\sigma$ can be chosen to be the special choice \eqref{eq:special-choice}. By setting $p=2$ in \eqref{eq:special-sing-equation}, cusps of $\mathcal{C}_w$ are the solutions of 
\[
M' - \pi^{1/2}M^2 = 0. 
\]
Noting that $M'$ has only even degree terms due to $p=2$, the above equation becomes 
\[
\Big((M')^{1/2} - \pi^{1/4}M\Big)^2 = 0.
\]
By \eqref{eq:sing-equation-degree}, the polynomial 
$
(M')^{1/2} - \pi^{1/4}M
$
is of degree $m = d-2$, and can be made separable for a general choice of $M$. The statement then follows from Lemma \ref{lem:sigularity-equation}. 
\end{proof}

Finally, we consider all remaining situations:

\begin{proposition}\label{prop:cusp-in-general}
Suppose $p > 2$ and $m\geq 1$. For a general choice of $\Delta_X$, the number of cusps on a general fiber of $\mathcal{C}_{\mathcal{U}} \to \mathcal{U}$ achieves the upper bound in \eqref{eq:cusp-bounds}. 
\end{proposition}
\begin{proof}
Let $\mathcal{C}_w$ be a fiber of $\mathcal{C}_{\mathcal{U}} \to \mathcal{U}$. We will show that \eqref{eq:special-sing-equation} for the special choice of $\Delta_X$, is separable for a general $\mathcal{C}_w$.  Since separability is an open condition, the general case \eqref{eq:singularity-equation} obtained as a deformation of \eqref{eq:special-sing-equation} by deforming $\Delta_X$ is again separable, which will conclude the proof. 

We now turn to prove \eqref{eq:special-sing-equation} is separable. Fix a non-zero coefficient $\pi^{1/p} \in \kk$, and consider 
\[
F(t,\lambda) := M' - \pi^{1/p}M^2\lambda^{p-2} \in \kk[t, \lambda]. 
\]
We check the curve $\Big(F(t, \lambda) = 0 \Big) \subset \AA^2$ is smooth for a general $M$. 

Otherwise, the following equations have a solution: 
\[
\begin{cases}
 F(t, \lambda) = M' - \pi^{1/p}M^2\lambda^{p-2} = 0 \\
 \frac{\partial F}{\partial t} = M'' - 2 \pi^{1/p}M M' \lambda^{p-2} \\
  \frac{\partial F}{\partial \lambda} = 2 \pi^{1/p}M^2 \lambda^{p-3}.
\end{cases}
\]
The third equation implies that the solution satisfies either $M = 0$ or $\lambda^{p-3} = 0$. If $M = 0$, by the first two equations the solution satisfy $M' = M'' = 0$, which cannot hold for a general $M$. Thus, the solution must satisfy $\lambda^{p-3} = 0$, hence $\lambda = 0$ and $p > 3$. In this case, the first two equation yield $M' = M'' = 0$, which is not true for a general $M$ either. 

Now for a fixed $\pi^{1/p} \neq 0$, we consider a line in $\AA^2$ given by the parameterization 
\begin{equation}\label{eq:separable-line}
(t, \lambda) = (t, V(t) - W(t)) = \Big( t, (b_0^{1/p} - c_0^{1/p}) + (b_1^{1/p} - c_1^{1/p})t \Big).
\end{equation}
Note that the solutions of \eqref{eq:special-sing-equation} are precisely given by the intersection of the line \eqref{eq:separable-line} with the smooth curve $\Big(F(t, \lambda) = 0 \Big) \subset \AA^2$. Note that the line \eqref{eq:separable-line} can be made general even with fixed $\pi^{1/p}\neq 0$. Thus the intersection of \eqref{eq:separable-line} and $\Big(F(t, \lambda) = 0 \Big) \subset \AA^2$ is a set of smooth points for a general choice of $V, W$ and $M$. 
This finishes the proof. 
\end{proof}

This finishes the proof of Theorem \ref{intro-thm:supercusps}. \qed

\newcommand{\etalchar}[1]{$^{#1}$}
\providecommand{\bysame}{\leavevmode\hbox to3em{\hrulefill}\thinspace}
\providecommand{\MR}{\relax\ifhmode\unskip\space\fi MR }
\providecommand{\MRhref}[2]{%
  \href{http://www.ams.org/mathscinet-getitem?mr=#1}{#2}
}
\providecommand{\href}[2]{#2}

\end{document}